\documentclass{siamart171218}

\ifpdf
\hypersetup{
  pdftitle={A new class of high-order methods for multirate differential equations},
  pdfauthor={V.~T.~Luan, R.~Chinomona, and D.~R.~Reynolds}
}
\fi
\usepackage{lipsum}
\usepackage{epstopdf}
\usepackage{graphicx}
\usepackage{array}
\usepackage{bigstrut}
\usepackage{epsfig}
\usepackage{amssymb}
\usepackage{amsmath} 		
\usepackage{amssymb} 		
\usepackage{amstext}
\usepackage{amsfonts}		
\usepackage{url}		    
\usepackage{hhline}			
\usepackage{array}
\usepackage{tabularx}
\usepackage{makecell}
\usepackage{algorithm}
\usepackage{algorithmic}
\usepackage{enumitem}
\usepackage{color}
\usepackage{caption}
\usepackage{sub caption}
\usepackage{enumitem}

\numberwithin{equation}{section}
\newcommand{\q}{\quad}
\newcommand{\ee}{{\rm e}\hspace{1pt}}
\newcommand{\dd}{\hspace{0.5pt}{\rm d}\hspace{0.5pt}}
\newcommand{\merkthree}{\texttt{MERK3}}
\newcommand{\merkfour}{\texttt{MERK4}}
\newcommand{\merkfive}{\texttt{MERK5}}
\newcommand{\mis}{\texttt{MIS-KW3}}

\ifpdf
  \DeclareGraphicsExtensions{.eps,.pdf,.png,.jpg}
\else
  \DeclareGraphicsExtensions{.eps}
\fi

\newsiamremark{remark}{Remark}
\newsiamremark{hypothesis}{Hypothesis}
\crefname{hypothesis}{Hypothesis}{Hypotheses}
\newsiamthm{claim}{Claim}

\headers{High-order methods for multirate differential equations}{V.~T.~Luan, R.~Chinomona, and D.~R.~Reynolds}

\title{A new class of high-order methods for multirate differential equations \thanks{Submitted to the editors DATE.
\funding{This work was supported in part by the U.S. Department of
  Energy, Office of Science, Office of Advanced Scientific Computing
  Research, Scientific Discovery through Advanced Computing (SciDAC)
  program through the FASTMath Institute under Lawrence Livermore
  National Laboratory Subcontract B626484.}}}

\author{Vu Thai Luan, Rujeko Chinomona, Daniel
  R. Reynolds\thanks{Department of Mathematics, Southern Methodist
    University, Dallas, TX 75275-0156 (\email{vluan@smu.edu},
    \email{rchinomona@smu.edu}, \email{reynolds@smu.edu}).}
}

\usepackage{amsopn}

\begin{document}

\maketitle

\begin{abstract}
This work focuses on the development of a new class of high-order
accurate methods for multirate time integration of systems of ordinary
differential equations.  The proposed methods are based on a specific
subset of explicit one-step exponential integrators.  More precisely,
starting from an explicit exponential Runge--Kutta method of the
appropriate form, we derive a multirate algorithm to approximate the
action of the matrix exponential through the definition of modified
``fast'' initial-value problems.  These fast problems may be solved
using any viable solver, enabling multirate simulations through use of
a subcycled method.  Due to this structure, we name these
\emph{Multirate Exponential Runge--Kutta} (MERK) methods.
In addition to showing how MERK methods may be derived, we provide
rigorous convergence analysis, showing that for an overall method of
order $p$, the fast problems corresponding to internal stages may be
solved using a method of order $p-1$, while the final fast problem
corresponding to the time-evolved solution must use a method of order
$p$.  Numerical simulations are then provided to demonstrate the
convergence and efficiency of MERK methods with orders three through
five on a series of multirate test problems.
\end{abstract}



\section{Introduction}
\label{section1}
In this paper, we focus on the construction, analysis, and implementation of
efficient, highly accurate, multirate time stepping algorithms, based
on various classes of explicit one-step exponential integrators. These
algorithms may be applied to initial value problems (IVPs) of the form
\begin{equation} \label{eq1}
  u'(t) = F(t,u(t)) = \mathcal{L}u(t) + \mathcal{N}(t,u(t)), \q  u(t_0)=u_0,
\end{equation}
on the interval $t_0< t \leq T$, where the vector field $F(t,u(t))$
can be decomposed into a linear part $\mathcal{L}u(t)$ comprising the ``fast''
time scale, and a nonlinear part $\mathcal{N}(t,u(t))$ comprising the ``slow''
time scale.  Such systems frequently result from so-called
``multi-physics'' simulations that couple separate physical processes
together, or from the spatial semi-discretization of time-dependent
partial differential equations (PDEs).  Our primary interest in this
paper lies in the case where the fast component is much less costly to
compute than the slow component, thereby opening the door for methods
that evolve each component with different time step sizes -- so-called
multirate (or multiple time-stepping, MTS) methods. This case is
common in practice when using a non-uniform grid for the spatial
semi-discretization of PDEs, or in parallel computations where the
fast component is comprised of spatially-localized processes but the
slow component requires communication across the parallel network.

In recent years, there has been renewed interest in the construction
of multirate time integration methods for systems of ODEs.
Generally, these efforts have focused on techniques to achieve orders
of accuracy greater than two, since second-order methods may be
obtained through simple interpolation between time scales.  These
recent approaches broadly fit into two categories: methods that attain
higher-order through extrapolation of low-order methods
\cite{Bouzarth2010, constantinescu2010, Constantinescu2013}, and methods
that directly satisfy order conditions for partitioned and/or additive
Runge--Kutta methods \cite{Fok2016, Gunther2001, Gunther2016,
  Knoth2012, Knoth2014, Kvaerno2000, Sandu2018, Sarshar2018,
  Schlegel2009, Schlegel2012a, Schlegel2012b, Sexton2018,
  Wensch2009}.
Of these, the latter category promises increased efficiency due to the
need to traverse the time interval only once.  However, only very
recently have methods of this type been constructed that can achieve
full fourth-order accuracy \cite{Sandu2018, Sexton2018}, and we know of
no previous methods having order five or higher.

Among numerical methods that use the same time step for all components
of \eqref{eq1}, exponential integrators have shown great promise
in recent years \cite{CM02, HL97, HLS98, HO05b, HO05a, HO10, LO13,LO14b, LO14a,LO16,Luan17,Luan18}.  
Most such methods require the approximation of products of
matrix functions with vectors, i.e., $\phi(\mathcal{L})\, v$, for
$\mathcal{L}\in \mathbb{R}^{d\times d}$ and $v\in \mathbb{R}^{d}$.

Inspired by recent results on local-time stepping methods for problems
related to \eqref{eq1} \cite{Gander2012, Grote2013b, Grote2010,
  Grote2013a}, and motivated by the idea in \cite[Sect.~5.3]{HO11}
that establishes a multirate procedure for exponential \emph{multistep}
methods of Adams-type, here we derive multirate procedures for exponential
\emph{one-step} methods.  Starting from an $s$-stage explicit exponential
Runge--Kutta (ExpRK) method applied to \eqref{eq1}, we employ the idea
of backward error analysis to define $s-1$ modified differential
equations whose exact solutions coincide with the ExpRK internal
stages.  These modified differential equations may then be
evolved using standard ODE solvers at the fast time scale.  We name
the resulting methods as \emph{Multirate Exponential Runge--Kutta}
(MERK) methods.

The ability to construct modified ODEs for each slow ExpRK stage is
dependent on the form of the ExpRK method itself, and we
identify these restrictions within this manuscript.
Using this approach, we derive a general multirate algorithm
(Algorithm \ref{alg2}) that can be interpreted as a
particular implementation (without matrix functions) of explicit
exponential Runge-Kutta methods.  With this algorithm in
hand, we perform a rigorous convergence analysis for the proposed
MERK methods.  We additionally construct MERK schemes with orders of
accuracy three through five, based on some well-known ExpRK methods
from the literature.

We note that the resulting methods show strong similarities to the
\emph{MIS} methods in \cite{Knoth2012, Knoth2014, Schlegel2009,
  Schlegel2012a, Schlegel2012b, Wensch2009} and the follow-on
\emph{RMIS} methods \cite{Sexton2018} and
\emph{MRI-GARK} methods \cite{Sandu2018}, in that the MERK algorithm
requires the construction of a set of modified ``fast'' initial-value
problems that must be solved to proceed between slow stages, and where
these modifications take the form of polynomials based on ``slow''
function data.  While these approaches indeed result in similar
algorithmic structure, (R)MIS and MRI-GARK methods are based on
partitioned and generalized-structure additive Runge--Kutta theory
\cite{Sandu2015}, and as such their derivation requires satisfaction
of many more order conditions than MERK methods, particularly as the
desired method order increases, to the end that no MIS method of order
greater than three, and no RMIS or MRI-GARK methods of order greater
than four, have ever been proposed.  Additionally, to obtain an overall
order $p$ method, all fast IVPs for (R)MIS and MRI-GARK methods must be
solved to order $p$, whereas the internal stages in MERK methods may
use an order $p-1$ solver. Finally, both (R)MIS and the MRI-GARK methods
require sorted abcissae $c_1\le c_2\le\cdots\le c_s$, a requirement
that is not present for MERK methods.

The outline of this paper is as follows: in Section~\ref{sec2}, we
derive the general class of exponential Runge-Kutta methods in a way
that facilitates construction of MERK procedures. In
Section~\ref{section3}, we then derive the general MERK algorithm for
exponential Runge--Kutta methods, and provide a rigorous convergence
analysis for these schemes. In Section~\ref{section4}, we derive
specific MERK methods based on existing exponential Runge--Kutta
methods.  We present a variety of numerical examples in
Section~\ref{sec6} to illustrate the efficiency of the new MERK
schemes with order of accuracy up to five.  The main contributions of
this paper are Algorithm~\ref{alg2}, convergence analysis for MERK
methods (Theorem \ref{theorem2}), and the proposed MERK schemes with
order of accuracy up to five.


\section{Motivation}
\label{sec2}
We begin with a general derivation of exponential Runge--Kutta methods \cite{HO05b,HO10}, which motivates a multirate procedure for solving \eqref{eq1}.
\subsection{Exponential Runge--Kutta methods}
\label{sec2.1}
When deriving ExpRK methods, it is crucial to represent the exact
solution of \eqref{eq1} at time $t_{n+1}=t_n +H$ using the
variation-of-constants formula,
\begin{equation} \label{eq3}
  u(t_{n+1})=u(t_n +H)=\ee^{H \mathcal{L}} u(t_n) + \int_{0}^{H} \ee^{(H-\tau)\mathcal{L}} \mathcal{N}(t_n+\tau, u(t_n+\tau)) \dd\tau.
\end{equation}
The integral in \eqref{eq3} is then approximated using a quadrature
rule having nodes $c_i$ and weights $b_i(H\mathcal{L})$ ($i=1,\ldots,s$), which
yields
\begin{equation} \label{eq3a}
  u(t_{n+1})\approx \ee^{H \mathcal{L}} u(t_n) + H\sum_{i=1}^{s} b_i(H\mathcal{L})\, \mathcal{N}(t_n+c_i H, u(t_n+c_i H)).
\end{equation}
By applying \eqref{eq3} (with $c_i H$ in place of $H$), the unknown
intermediate values $u(t_n+c_i H)$ in \eqref{eq3a} can be represented
as
\begin{equation} \label{eq3b}
  u(t_n+c_i H)=\ee^{c_i H \mathcal{L}}u(t_n) + \int_{0}^{c_i H} \ee^{(c_i H-\tau)\mathcal{L}} \mathcal{N}(t_n+\tau, u(t_n+\tau)) \dd\tau.
\end{equation}
Again, one can use another quadrature rule with the same nodes $c_i$
as before (to avoid the generation of new unknowns) and new weights
$a_{ij}(H\mathcal{L})$ to approximate the integrals in \eqref{eq3b}. This gives
\begin{equation} \label{eq3c}
  u(t_n+c_i H)\approx \ee^{c_i H \mathcal{L}}u(t_n) + H\sum_{j=1}^{s} a_{ij}(H\mathcal{L})\,\mathcal{N}(t_n+c_j H, u(t_n+c_j H)).
\end{equation}
Now, denoting the approximations $u_n \approx u(t_n)$ and $U_{n,i}
\approx u(t_n+c_i H)$, then from \eqref{eq3a} and \eqref{eq3c} one may
obtain the so-called \emph{exponential Runge--Kutta methods} 
\begin{subequations} \label{eq4}
\begin{align}
  U_{n,i}&= \ee^{c_i H  \mathcal{L}}u_n +H \sum_{j=1}^{s}a_{ij}(H \mathcal{L})\, \mathcal{N}(t_n +c_j H, U_{n,j}), \q  i=1,\ldots,s,  \label{eq4a} \\
  u_{n+1} &= \ee^{H \mathcal{L}}u_n +  H \sum_{i=1}^{s}b_{i}(H \mathcal{L})\, \mathcal{N}(t_n +c_i H, U_{n,i}).  \label{eq4b}
\end{align}
\end{subequations}

The formula \eqref{eq4} is considered {\em explicit} when
$a_{ij}(H \mathcal{L})=0$ for all $i\leq j$ (thus $c_1=0$ and
consequentially $U_{n,1}=u_n$).  Throughout this paper we restrict our
attention to explicit exponential Runge--Kutta methods,
which can be reformulated as (see \cite{LO14b,LO12b}):
\begin{subequations} \label{eq:expRK}
\begin{align}
  U_{n,i}&= u_n + c_i H \varphi _{1} ( c_i H \mathcal{L})F(t_n, u_n) +
          H \sum_{j=2}^{i-1}a_{ij}(H \mathcal{L}) D_{n,j}, \  i=2,\ldots,s,  \label{eq:expRKa} \\
  u_{n+1}& = u_n + H \varphi _{1} ( H \mathcal{L})F(t_n, u_n) + H \sum_{i=2}^{s}b_{i}(H \mathcal{L}) D_{n,i}  \label{eq:expRKb},
\end{align}
\end{subequations}
where
\begin{equation} \label{eq:Dni}
  D_{n,i}= \mathcal{N}(t_n+c_i H, U_{n,i})- \mathcal{N}(t_n, u_n ), \qquad  i=2,\ldots,s.
\end{equation}
Here, the coefficients $a_{ij}(H \mathcal{L})$ and $b_{i}(H \mathcal{L}) $ are often
linear combinations of the functions $\varphi _{k} (c_i H \mathcal{L})$ and
$\varphi_{k} (H \mathcal{L})$, respectively, wherein $\varphi_k (z)$ are given by
\begin{equation} \label{eq8}
  \varphi_{k}(z)=\int_{0}^{1} \ee^{(1-\theta )z} \frac{\theta^{k-1}}{(k-1)!}\dd\theta , \quad k\geq 1,
\end{equation}
and satisfy the recurrence relations
\begin{equation} \label{eq9}
  \varphi_{k}(z)=\frac{\varphi_{k-1}(z)-\varphi_{k-1}(0)}{z}, \q  \varphi_{0}(z)=\ee^z.
\end{equation}


\subsection{Adopting the idea of backward error analysis}
\label{sec2.2}
Motivated by the idea of \cite[Sect. 5.3]{HO11}, and recalling the
equations \eqref{eq3} and \eqref{eq3b}, we note that $u(t_{n+1})$ and
$u(t_n +c_i H)$ are the exact solutions of the differential equation
\begin{equation}  \label{eq6}
  v'(\tau)=\mathcal{L} v(\tau) + \mathcal{N}(t_n+\tau,u(t_n+\tau)), \q  v(0)=u(t_n),
\end{equation}
evaluated at $\tau=H$ and $\tau=c_i H$, respectively. In other words,
solving \eqref{eq6} exactly (by means of using the
variation-of-constants formula) on the time intervals $[0, H]$ and
$[0, c_i H]$ shows that $v(H)=u(t_{n+1})$ and $v(c_i H)=u(t_n +c_i H)$.
Unfortunately, explicit representations of these analytical solutions
are generally impossible to find, since $u(t_n)$ and $u(t_n+\tau)$ are
unknown values.  This observation, however, suggests the use of
backward error analysis (see, for instance \cite[Chap.~IX]{HLW06}).

Given an exponential Runge-Kutta method \eqref{eq4}, we therefore
search for modified differential equations of the form \eqref{eq6},
such that their exact solutions at $\tau=c_i H$ and $\tau=H$ coincide
with the ExpRK approximations $U_{n,i}$ ($i=2,\ldots,s$) and
$u_{n+1}$, respectively.  We may then approximate solutions to these
modified equations to compute our overall approximation of \eqref{eq1}.


\section{Multirate exponential Runge--Kutta methods}
\label{section3}
In this section, we construct a new multirate procedure based on
approximation of ExpRK schemes; we call the resulting algorithms
\emph{Multirate Exponential Runge-Kutta} (MERK) methods.  Following
this derivation, we present a rigorous stability and convergence
analysis.
\subsection{Construction of modified differential equations}
\label{section3.1}
We begin with the construction of MERK methods, through definition of
modified differential equations corresponding with the ExpRK stages
$U_{n,i}$ ($i=2,\ldots,s$) and solution $u_{n+1}$.
\begin{theorem}\label{theorem1}
  Assuming that the coefficients $a_{ij}(H \mathcal{L})$ and $b_{i}(H
  \mathcal{L})$ of an explicit exponential Runge--Kutta method
  \eqref{eq:expRK} may be written as linear combinations
  \begin{equation} \label{eq7}
    a_{ij}(H \mathcal{L})=\sum_{k=1}^{\ell_{ij}}\alpha^{(k)}_{ij}\varphi_{k}(c_i H\mathcal{L}), \q
    b_{i}(H \mathcal{L})=\sum_{k=1}^{m_i}\beta^{(k)}_{i}\varphi_{k}(H \mathcal{L})
  \end{equation}
  for some positive integers $\ell_{ij}$ and $m_i$, and where the functions
  $\varphi _{k} (c_i H \mathcal{L})$ and $\varphi_{k} (H \mathcal{L})$
  are given in \eqref{eq8}, then $U_{n,i}$ and $u_{n+1}$  are the
  exact solutions of the linear differential equations
  \begin{subequations} \label{eq14}
    \begin{align}
      v'_{n,i}(\tau)&=\mathcal{L}v_{n,i}(\tau) +   p_{n,i}(\tau), && v_{n}(0)=u_n, \qquad  i=2,\ldots,s, \label{eq14a} \\
      v'_n(\tau)&=\mathcal{L}v_n(\tau)  +   q_{n}(\tau),  &&  v_n(0)=u_n \hspace{2.3cm} \label{eq14b}
    \end{align}
  \end{subequations}
  at the times $\tau= c_i H$ and $\tau= H$, respectively. Here
  $p_{n,i}(\tau)$ and $q_{n}(\tau)$ are polynomials in $\tau$ given by
  \begin{subequations} \label{eq13}
    \begin{align}
      p_{n,i}(\tau)&= \mathcal{N}(t_n, u_n)+\sum_{j=2}^{i-1} \Big(\sum_{k=1}^{\ell_{ij}}\dfrac{\alpha^{(k)}_{ij}}{c^k_i H^{k-1} (k-1)!}\tau^{k-1}\Big) D_{n,j}, \label{eq13a} \\
      q_{n}(\tau) &= \mathcal{N}(t_n, u_n)+\sum_{i=2}^{s} \Big(\sum_{k=1}^{m_i}\dfrac{\beta^{(k)}_{i} }{H^{k-1}(k-1)!}\tau^{k-1} \Big) D_{n,i}.  \label{eq13b}
    \end{align}
  \end{subequations}
\end{theorem}
\begin{proof}
  By changing the integration variable to $\tau=H\theta$ in \eqref{eq8}, we obtain
  \begin{equation} \label{eq10}
    \varphi _{k}(z)=\frac{1}{H^k}\int_{0}^{H} \ee^{(H-\tau)\frac{z}{H}} \frac{\tau^{k-1}}{(k-1)!}\dd\tau , \quad k\geq 1.
  \end{equation}
  Substituting $z=c_i H\mathcal{L}$ and $z=H\mathcal{L}$ into \eqref{eq10} and inserting the obtained results for $\varphi _{k} (c_i H \mathcal{L})$ and $\varphi_{k} (H \mathcal{L})$ into \eqref{eq7}  shows that
  \begin{subequations}\label{eq11}
    \begin{align}
      a_{ij}(H \mathcal{L})&=\int_{0}^{c_i H} \ee^{(c_i H-\tau)\mathcal{L}} \sum_{k=1}^{\ell_{ij}}\dfrac{\alpha^{(k)}_{ij}}{(c_i H)^{k} (k-1)!}\tau^{k-1}\dd\tau, \label{eq11a}\\
      b_{i}(H \mathcal{L})&=\int_{0}^{H} \ee^{(H-\tau)\mathcal{L}} \sum_{k=1}^{m_i}\dfrac{\beta^{(k)}_{i}}{H^{k}(k-1)!}\tau^{k-1}\dd\tau. \label{eq11b}
    \end{align}
  \end{subequations}
  Using the fact that $F(t_n, u_n)=\mathcal{L}u_n +\mathcal{N}(t_n, u_n)$ and
  \begin{equation} \label{eq:varphi1}
    \varphi _{1} (Z)=\frac{1}{H}\int_{0}^{H} \ee^{(H-\tau)\frac{Z}{H}} \dd\tau=(\ee^Z- I) Z^{-1}.
  \end{equation}
  we can write \eqref{eq:expRK} in an equivalent form,
  \begin{subequations} \label{eq:expRKnew}
    \begin{align}
      U_{n,i}&=  \ee^{c_i  H \mathcal{L}} u_n + c_i H \varphi _{1} ( c_i H \mathcal{L})\mathcal{N}(t_n, u_n) +
              H \sum_{j=2}^{i-1}a_{ij}(H \mathcal{L}) D_{n,j},  \label{eq:expRKnewa} \\
      u_{n+1}& = \ee^{H \mathcal{L}}u_n+ H \varphi _{1} ( H \mathcal{L})\mathcal{N}(t_n, u_n) + H \sum_{i=2}^{s}b_{i}(H \mathcal{L}) D_{n,i}  \label{eq:expRKnewb},
    \end{align}
  \end{subequations}
  for $i=2,\ldots,s$.  We now insert the integral form of
  $\varphi_1(Z)$ in \eqref{eq:varphi1} (with $Z= c_i H \mathcal{L}$
  and $Z=H \mathcal{L}$) and \eqref{eq11} into \eqref{eq:expRKnew} to
  get
  \begin{subequations} \label{eq12}
    \begin{align}
      U_{n,i}&= \ee^{c_i  H \mathcal{L}} u_n +\int_{0}^{c_i H} \ee^{(c_i H-\tau)\mathcal{L}}   p_{n,i}(\tau) \dd\tau, \q  i=2,\ldots,s,  \label{eq12a} \\
      u_{n+1} &= \ee^{H \mathcal{L}}u_n +  \int_{0}^{H} \ee^{(h-\tau)\mathcal{L}}  q_{n}(\tau) \dd\tau  \label{eq12b}
    \end{align}
  \end{subequations}
  with $p_{n,i}(\tau)$ and $q_{n}(\tau)$ as shown in \eqref{eq13}.
  Clearly, these representations (variation-of-constant formulas)
  show the conclusion of Theorem~\ref{theorem1}.  In particular,
  $U_{n,i}=v_{n,i}(c_i H)$ and $u_{n+1}=v_n(H)$. Thus one can consider
  \eqref{eq14} as modified differential equations with identical
  solutions as the ExpRK approximations to \eqref{eq6}.
\end{proof}
We note that the idea of using an ODE to represent a linear combination of matrix-vector $\varphi_{k}(A)v_k$ was also used in \cite{Niesen2012}.
\subsection{MERK methods and a multirate algorithm}
Clearly, the polynomials \\
$p_{n,i}(\tau)$ and $q_{n}(\tau)$ in
\eqref{eq13} are not given analytically since $D_{n,i}$ are unknowns;
however, these polynomials can be numerically determined as
follows.  For simplicity, we illustrate our procedure by starting with
$n=0$ and $i=s=2$.  In this case we know $u_0=u(t_0)$ and
$p_{0,2}(\tau)=\mathcal{N}(t_0, u_0)$, so one can solve the ODE
\eqref{eq14a} on $[0, c_2 H]$ to get an approximation to $U_{0,2}$,
$\widehat{U}_{0,2}\approx U_{0,2}=v_{0,2}(c_2 H)$.  Then replacing the
unknown $U_{0,2}$ in \eqref{eq13b} by $\widehat{U}_{0,2}$, we have
\[
  \hat{q}_0(\tau)=\mathcal{N}(t_0, u_0)+\sum_{k=1}^{m_i}\dfrac{\beta^{(k)}_{2} }{H^{k-1}(k-1)!}\tau^{k-1}
  \widehat{D}_{0,2},
\]
where $\widehat{D}_{0,2}=\mathcal{N}(t_0+c_2 H, \widehat{U}_{0,2})-
\mathcal{N}(t_0, u_0)$.  Since $\hat{q}_0(\tau) \approx q_0(\tau)$, we
may then solve the ODE \eqref{eq14b} on $[0, H]$ with $\hat{q}_0(\tau)$ in
place of $q_0(\tau)$ to obtain an approximation $\hat{u}_1 \approx u_1=v_0 (H)$.

This general process may be extended to larger numbers of stages $s\ge
2$ and for subsequent time steps $n\ge 0$.  Approximating $\hat{u}_n \approx
u_n$ (with $\hat{u}_0=u_0$), then for $i=2,\ldots,s$, we define the
following perturbed linear ODEs over $\tau \in [0, c_i H]$:
\begin{equation} \label{eq13nc}
  y'_{n,i}(\tau)=\mathcal{L}y_{n,i}(\tau) + \hat{p}_{n,i}(\tau), \q  y_{n,i}(0)=\hat{u}_n,
\end{equation}
with
\begin{align}
  \label{eq:p_hat}
  \hat{p}_{n,i}(\tau) &= \mathcal{N}(t_n, \hat{u}_n)+\sum_{j=2}^{i-1} \Big(\sum_{k=1}^{\ell_{ij}}\dfrac{\alpha^{(k)}_{ij}}{c^k_i H^{k-1} (k-1)!}\tau^{k-1}\Big) \widehat{D}_{n,j},\\
  \label{eq:hatDni}
  \widehat{D}_{n,i} &= \mathcal{N}(t_n+c_i H, \widehat{U}_{n,i}) - \mathcal{N}(t_n, \widehat{u}_n ),
\end{align}
that provide the approximations
\[
  \widehat{U}_{n,i}\approx y_{n,i}(c_i H) \approx v_{n,i}(c_i H)=U_{n,i}.
\]
With these in place, we then solve the linear ODE
\begin{equation} \label{eq13nd}
  y'_{n}(\tau)=\mathcal{L}y_{n}(\tau) +   \hat{q}_{n}(\tau), \q  y_{n}(0)=\hat{u}_n
\end{equation}
over $\tau\in [0, H]$, with
\begin{equation} \label{eq:q_hat}
  \hat{q}_{n}(\tau) = \mathcal{N}(t_n, \hat{u}_n)+\sum_{i=2}^{s} \Big(\sum_{k=1}^{m_i}\dfrac{\beta^{(k)}_{i} }{H^{k-1}(k-1)!}\tau^{k-1} \Big) \widehat{D}_{n,i},
\end{equation}
to obtain the approximate time-step solutions,
\[
  \hat{u}_{n+1}\approx y_{n}(H) \approx v_{n}(H)=u_{n+1}.
\]

Since the above procedure uses a ``macro'' time step $H$ to integrate
the slow process, and a ``micro'' time step $h$ to integrate the fast
process (via solving the ODEs \eqref{eq13nc} and \eqref{eq13nd}), we
call the resulting methods \eqref{eq13nc}-\eqref{eq:hatDni}
\emph{Multirate Exponential Runge--Kutta (MERK)} methods.  By
construction, these MERK methods offer several interesting features.
They reduce the solution of nonlinear problems \eqref{eq1} to the
solution of a sequence of linear differential equations \eqref{eq13nc} and
\eqref{eq13nd}, using very few evaluations of the nonlinear operator
$\mathcal{N}$. Thus they can be more efficient for problems where the
linear part is much less costly to compute than the nonlinear part.
Additionally, they do not require the computation of matrix functions,
as is the case with ExpRK methods.  Moreover, these methods do not
require a starting value procedure as in multirate algorithms for
exponential multistep methods \cite{Demirel2015,HO11}.

We provide the following Algorithm~\ref{alg2} to give a succinct
overview of the implementation of our MERK methods.
\begin{algorithm}[h]
\caption{MERK method}
\label{alg2}
\begin{list}{$\bullet $}{}
\item \textbf{Input:}  $\mathcal{L}$; $\mathcal{N}(t,u)$; $t_0$; $u_0$; $s$; $c_i$ ($i=1,\ldots,s$); $H$
\item \textbf{Initialization:}  Set $n=0$; $\hat{u}_n=u_0$.\\
While $t_n<T$
\begin{enumerate}
  \item Set $\widehat{U}_{n,1}=\hat{u}_n$.
\item For $i=2,\ldots,s$ do
\begin{enumerate}
  \item Find  $\hat{p}_{n,i}(\tau)$ as in \eqref{eq:p_hat}.
  \item Solve \eqref{eq13nc} on $[0, c_i H]$ to obtain $\widehat{U}_{n,i}\approx y_{n,i}(c_i H)$.
\end{enumerate}
 \item Find $\hat{q}_{n,s}(\tau)$ as in \eqref{eq:q_hat}
  \item Solve \eqref{eq13nd} on $[0, H]$ to get $\hat{u}_{n+1}\approx y_{n}(H).$
  \item Update $t_{n+1}:=t_n+H$, $n:=n+1$.
\end{enumerate}
\item \textbf{Output:} Approximate values $\hat{u}_n\approx u_n, n=1,2,\ldots$ (where
$u_n$ is the numerical solution at time $t_n$ obtained by an ExpRK method).
\end{list}
\end{algorithm}
\subsection{Stability and convergence analysis}
\label{sec:analysis}
Since MERK methods are constructed to approximate ExpRK methods, we
perform their error analysis in the framework of analytic semigroups
on a Banach space $X$, under the following assumptions (see
e.g., \cite{HO05a,LO14b}).

{\em Assumption 1. The linear operator $\mathcal{L}$ is the
  infinitesimal generator of an analytic semigroup
  $\ee^{t\mathcal{L}}$ on $X$}. This implies that
\begin{equation} \label{eq:bound1}
\|\ee^{t\mathcal{L}}\|_{X\leftarrow X}\leq C, \quad t\geq 0
\end{equation}
and consequently $\varphi_k(H \mathcal{L})$,   $a_{ij}(H \mathcal{L})$
and $b_{i}(H \mathcal{L})$  are bounded operators. \\

{\em Assumption 2 (for high-order methods). The solution $u:[t_0,
  T]\to X$  of \eqref{eq1} is sufficiently smooth with derivatives in
  $X$, and $\mathcal{N}:[t_0, T]\times X \to X$ is sufficiently
  Fr\'echet differentiable in a strip along the exact solution.}  All
  derivatives occurring in the remainder of this section are therefore
  assumed to be uniformly bounded.\\

We analyze the error in MERK methods starting with the local error of
ExpRK methods. Therefore, we first consider \eqref{eq4} (in its
explicit form) with exact initial value, $u_n=u(t_n)$:
\begin{subequations} \label{eq3.17}
\begin{align}
\breve{U}_{n,i}&= \ee^{c_i  H \mathcal{L}} u(t_n) +H \sum_{j=1}^{i-1}a_{ij}(H \mathcal{L})\, \mathcal{N}(t_n+c_j H, \breve{U}_{n,j}), \q  i=2,\ldots,s,  \label{eq3.17a} \\
\breve{u}_{n+1} &= \ee^{H \mathcal{L}}u(t_n) +  H \sum_{i=1}^{s}b_{i}(H \mathcal{L})\, \mathcal{N}(t_n+c_i H, \breve{U}_{n,i})  \label{eq3.17b}
\end{align}
\end{subequations}
and thus the MERK methods  \eqref{eq14}--\eqref{eq13} are considered with polynomials
\begin{subequations} \label{eq3.18}
\begin{align}
 \breve{p}_{n,i}(\tau)&= \mathcal{N}(t_n, u(t_n))+\sum_{j=2}^{i-1} \Big(\sum_{k=1}^{\ell_{ij}}\dfrac{\alpha^{(k)}_{ij}}{c^k_i H^{k-1} (k-1)!}\tau^{k-1}\Big) \breve{D}_{n,j}, \label{eq3.18a} \\
 \breve{q}_{n}(\tau) &= \mathcal{N}(t_n, u(t_n))+\sum_{i=2}^{s} \Big(\sum_{k=1}^{m_i}\dfrac{\beta^{(k)}_{i} }{H^{k-1}(k-1)!}\tau^{k-1} \Big) \breve{D}_{n,i},  \label{eq3.18b}
\end{align}
\end{subequations}
where $\breve{D}_{n,i}= \mathcal{N}(t_n+c_i H, \breve{U}_{n,i})- \mathcal{N}(t_n, u(t_n) )$.\\

\noindent \textbf{Error notation.}
Since MERK methods consist of approximations to approximations, we
must clearly isolate the errors induced at each approximation level.
To this end, we let $\hat{e}_{n+1} = \hat{u}_{n+1} - u(t_{n+1})$ denote
the global error at time $t_{n+1}$ of a MERK method
\eqref{eq13nc}-\eqref{eq:hatDni}.  Let $\breve{e}_{n+1} =
\breve{u}_{n+1} - u(t_{n+1})$ denote the local error at $t_{n+1}$ of
the base ExpRK method.  Let
$\hat{\varepsilon}_{n,i}=\widehat{U}_{n,i}- y_{n,i}(c_i H)$ and
$\hat{\varepsilon}_{n+1} =\hat{u}_{n+1} -y_n (H)$ denote the (global)
errors of the ODE solvers when integrating \eqref{eq13nc} on $[0, c_i
H]$ and \eqref{eq13nd} on $[0, H]$ (note that
$\hat{\varepsilon}_{n,1}=\hat{u}_n - y_{n,i}(0)=0$ since $c_1=0$).

First, we may write total error as the sum of the errors in each approximation,
\begin{equation} \label{eq3.19}
  \hat{e}_{n+1} =
  \hat{\varepsilon}_{n+1}+ (y_n (H) -\breve{u}_{n+1})+ \breve{e}_{n+1}.
\end{equation}
Applying the variation-of-constants formula to
\eqref{eq:q_hat} and using \eqref{eq11b}, we then write
\begin{equation} \label{eq3.20}
  y_n(H)=\ee^{H \mathcal{L}}\hat{u}_n + H \sum_{i=1}^{s}b_{i}(H \mathcal{L})\, \mathcal{N}(t_n+c_i H, \widehat{U}_{n,i}).
\end{equation}
Inserting $y_n (H)$ and $\breve{u}_{n+1}$ from \eqref{eq3.20} and
\eqref{eq3.17b} into \eqref{eq3.19} gives
\begin{equation} \label{eq3.21}
  \hat{e}_{n+1} = \ee^{H \mathcal{L}}\hat{e}_n +\hat{\varepsilon}_{n+1}+  H \mathcal{S}_{n,s} +  \breve{e}_{n+1},
\end{equation}
where
\begin{equation} \label{eq3.22}
  \mathcal{S}_{n,s} =\sum_{i=1}^{s}b_{i}(H \mathcal{L})\big(\mathcal{N}(t_n+c_i H, \widehat{U}_{n,i}) -\mathcal{N}(t_n+c_i H, \breve{U}_{n,i}) \big).
\end{equation}
Next, we prove some preliminary results.
\begin{lemma} \label{lemma3.2}
Denoting  $\widehat{E}_{n,i}=\widehat{U}_{n,i}-\breve{U}_{n,i}$
and $\breve{N}_{n,i}=\frac{\partial \mathcal{N}}{\partial u}(t_n
+ c_i H, \breve{U}_{n,i})$,   we have
\begin{equation} \label{eq3.23}
  \widehat{E}_{n,i}=\hat{\varepsilon}_{n,i}+  \ee^{c_i  H \mathcal{L}} \hat{e}_n + H \sum_{j=1}^{i-1}a_{ij}(H\mathcal{L}) (\breve{N}_{n,j} \widehat{E}_{n,j} + \widehat{R}_{n,j})
\end{equation}
with
\begin{equation} \label{eq3.23a}
  \widehat{R}_{n,j} =\int_{0}^{1} (1-\theta ) \frac{\partial^2  \mathcal{N}}{\partial u^2}(t_n + c_j H, \breve{U}_{n,j} + \theta \widehat{E}_{n,j})(\widehat{E}_{n,j}, \widehat{E}_{n,j})\dd\theta.
\end{equation}
Furthermore, under Assumption~2, the bound
\begin{equation} \label{eq3.24}
  \|\widehat{R}_{n,j}\|\leqslant C \|\widehat{E}_{n,j}\|^2, \ \text{i.e.},  \ \widehat{R}_{n,j}=\mathcal{O}(\|\widehat{E}_{n,j}\|^2)
\end{equation}
is held as long as $\widehat{E}_{n,j}$ remains in a sufficiently small
neighborhood of $0$.
\end{lemma}
\begin{proof}
We first rewrite
\begin{equation} \label{eq3.25}
  \widehat{E}_{n,i}=\widehat{U}_{n,i}- y_{n,i}(c_i H) + (y_{n,i}(c_i H) -\breve{U}_{n,i})=\hat{\varepsilon}_{n,i}+(y_{n,i}(c_i H) -\breve{U}_{n,i}).
\end{equation}
Here $y_{n,i}(c_i H)$ is the exact solution of  \eqref{eq13nc}, which
can be represented by the variation-of-constants formula and then
rewritten by using \eqref{eq11a} and \eqref{eq:p_hat} as:
\begin{equation} \label{eq3.26}
  y_{n,i}(c_i H)=\ee^{c_i  H \mathcal{L}} \hat{u}_n + H \sum_{j=1}^{i-1}a_{ij}(H \mathcal{L})\mathcal{N}(t_n+c_j H, \widehat{U}_{n,j}).
\end{equation}
Subtracting \eqref{eq3.17a} from  \eqref{eq3.26} and inserting the
obtained result into \eqref{eq3.25} gives
\begin{equation} \label{eq3.27}
  \widehat{E}_{n,i}=\hat{\varepsilon}_{n,i}+ \ee^{c_i  H \mathcal{L}} \hat{e}_n + H \sum_{j=1}^{i-1}a_{ij}(H \mathcal{L}) \big(\mathcal{N}(t_n+c_j H, \widehat{U}_{n,j}) -\mathcal{N}(t_n+c_j H, \breve{U}_{n,j}) \big).
\end{equation}
Using the Taylor series expansion of $\mathcal{N}(t, u)$ at $(t_n +c_j H, \breve{U}_{n,j})$, we get
\begin{equation} \label{eq3.28}
  \mathcal{N}(t_n+c_j H, \widehat{U}_{n,j}) -\mathcal{N}(t_n+c_j H, \breve{U}_{n,j})=\breve{N}_{n,j} \widehat{E}_{n,j} + \widehat{R}_{n,j}
\end{equation}
with the remainder  $\widehat{R}_{n,j}$ given in \eqref{eq3.23a},
which clearly satisfies \eqref{eq3.24} due to Assumption~2.
Inserting \eqref{eq3.28} into \eqref{eq3.27} shows \eqref{eq3.23}.
\end{proof}
\begin{lemma} \label{lemma3.3}
Under Assumptions 1 and 2,  there exist bounded operators
$\mathcal{T}_{n,i}(\hat{\varepsilon}_{n,i})$ and $\mathcal{B}_{n}
(\hat{e}_n)$ on $X$ such that
\begin{equation} \label{eq:Sni}
  \mathcal{S}_{n,s} =\sum_{i=2}^{s}\big(b_{i}(H \mathcal{L})\breve{N}_{n,i}+ H\mathcal{T}_{n,i}(\hat{\varepsilon}_{n,i}) \big)\hat{\varepsilon}_{n,i}+ \mathcal{B}_{n} (\hat{e}_n)\hat{e}_n.
\end{equation}
Note that  $\mathcal{T}_{n,i}$ also depends on $H$,
$\hat{\varepsilon}_{n,j}$, $a_{ij}(H \mathcal{L}) $,
$\breve{N}_{n,j}$ ($j=2,\ldots,i-1$), and $\hat{e}_{n}$; and
$\mathcal{B}_{n}$ also depends on $H$, $b_{i}(H \mathcal{L}), a_{ij}(H
\mathcal{L}) $,  $c_i$, and $\breve{N}_{n,i}$.
\end{lemma}
\begin{proof}
Inserting \eqref{eq3.28}  (with $i$ in place of $j$) into \eqref{eq3.22} gives
\begin{equation} \label{eq3.29}
\mathcal{S}_{n,s}  =\sum_{i=1}^{s}b_{i}(H \mathcal{L})\big(\breve{N}_{n,i} \widehat{E}_{n,i} + \widehat{R}_{n,i}\big).
\end{equation}
Using the recursion \eqref{eq3.23} from Lemma~\ref{lemma3.2}, we
further expand $\widehat{E}_{n,i}$ as
\begin{equation} \label{eq3.30}
\begin{aligned}
  \widehat{E}_{n,i}=\hat{\varepsilon}_{n,i} &+  H \sum_{j=1}^{i-1}a_{ij}(H\mathcal{L}) \breve{N}_{n,j} \hat{\varepsilon}_{n,j} +  H^2 \sum_{j=1}^{i-1}a_{ij}(H\mathcal{L}) \breve{N}_{n,j}  \sum_{k=1}^{j-1}a_{jk}(H\mathcal{L}) \breve{N}_{n,k} \widehat{E}_{n,k} \\
  &+H\sum_{j=1}^{i-1}a_{ij}(H \mathcal{L}) \widehat{R}_{n,j}  + H^2 \sum_{j=1}^{i-1}a_{ij}(H\mathcal{L}) \breve{N}_{n,j}  \sum_{k=1}^{j-1}a_{jk}(H\mathcal{L}) \widehat{R}_{n,k} \\
  \\
  &+ \Big( \ee^{c_i  H \mathcal{L}} + H \sum_{j=1}^{i-1}a_{ij}(H\mathcal{L}) \breve{N}_{n,j} \ee^{c_j  H \mathcal{L}} \Big)\hat{e}_{n}.
\end{aligned}
\end{equation}
Using \eqref{eq3.23} and \eqref{eq3.24}
($\widehat{R}_{n,i}=\mathcal{O}(\|\widehat{E}_{n,i}\|^2$) and
proceeding by induction, one can complete the recursion \eqref{eq3.30}
for $\widehat{E}_{n,i}$.  Inserting this recursion into \eqref{eq3.29}
(and noting that $\hat{\varepsilon}_{n,1}=0$) yields
\eqref{eq:Sni}. Based on the structure of \eqref{eq3.30} and
\eqref{eq3.29}, under the given assumptions it is clear that the
boundedness of $\mathcal{T}_{n,i}(\hat{\varepsilon}_{n,i})$ and
$\mathcal{B}_{n} (\hat{e}_n)$ follow from the boundedness of
$a_{ij}(H \mathcal{L}), b_{i}(H \mathcal{L})$, $\breve{N}_{n,i}$, and
$\widehat{R}_{n,i}$.
\end{proof}
We now present the main convergence result for MERK methods.
\begin{theorem}\label{theorem2}
Let the initial value problem \eqref{eq1} satisfy Assumptions
1--2.  Consider for its numerical solution a MERK method
\eqref{eq13nc}--\eqref{eq:hatDni} that is constructed from an ExpRK
method of global order $p$.
We further assume that the ``fast'' ODEs \eqref{eq13nc} and \eqref{eq13nd}
associated with the MERK method are integrated with micro time step
$h=H/m$ by using ODE solvers that have global order of convergence $q$ and
$r$, respectively, and where $m$ is the number of fast steps per slow
step. Then, the MERK method is convergent, and has error bound
\begin{equation}\label{eq3.31}
  \| u_n -u(t_n) \| \leq C_1 H^p + C_2 Hh^q + C_3h^r
\end{equation}
on compact time intervals \ $t_0 \leq  t_n =t_0+nH \leq  T$. Here,
the constant $C_1$ depends on $T-t_0$, but is independent of $n$ and
$H$; and the constants $C_2$ and $C_3$ also depend on the error
constants of the choosen ODE solvers.
\end{theorem}
\begin{proof}
We first note that since we only employ the fast ODE solvers on
time intervals $[0,c_iH]$ and $[0,H]$, then our assumption regarding
their accuracies of order $q$ and $r$ is typically
equivalent to \cite[Thm.~3.6]{hairer93}
\begin{subequations} \label{fastError}
  \begin{align}
    \label{fastError.a}
    \hat{\varepsilon}_{n,i}
    &= \frac{\tilde{c}_2}{\Lambda_i} h^q \left(\ee^{\Lambda_i c_i H}-1\right)
      = \tilde{c}_2 c_i \varphi_1 (\Lambda_i c_i H) H h^q
      = c_2 H h^q,\\
    \label{fastError.b}
    \hat{\varepsilon}_n
    &= \frac{\tilde{c}_3}{\Lambda} h^r \left(\ee^{\Lambda H}-1\right)
       = \tilde{c}_3 \varphi_1 (\Lambda H) H h^r
      = c_3 H h^r,
  \end{align}
\end{subequations}
(due to \eqref{eq:varphi1}) where
$\Lambda_i,\Lambda$ are the Lipschitz constants for the increment
functions of the ODE solvers applied to the problems \eqref{eq13nc}
and \eqref{eq13nd}, respectively.

For simplicity of notation, we denote $B_{n,i}=b_{i}(H \mathcal{L})\breve{N}_{n,i}+
H\mathcal{T}_{n,i}(\hat{\varepsilon}_{n,i})$. \\
Clearly, $B_{n,i}$ is a bounded operator and thus \eqref{eq:Sni} becomes
\begin{equation} \label{eq3.32}
  \mathcal{S}_{n,s} =\sum_{i=2}^{s} B_{n,i}\hat{\varepsilon}_{n,i}+ \mathcal{B}_{n} (\hat{e}_n)\hat{e}_n.
\end{equation}
Inserting this into \eqref{eq3.21} gives
\begin{equation} \label{eq3.33}
  \hat{e}_{n+1} = \ee^{H \mathcal{L}}\hat{e}_n + H\mathcal{B}_{n} (\hat{e}_n)\hat{e}_n +\breve{e}_{n+1}+ H\left(\sum_{i=2}^{s} B_{n,i}\hat{\varepsilon}_{n,i}+ \hat{\varepsilon}_{n+1}\right).
\end{equation}
Solving recursion \eqref{eq3.33} and using $\hat{e}_0=0$ (since
$\hat{u}_0 = u_0= u(t_0)$) finally yields
\begin{equation} \label{eq3.34}
  \hat{e}_{n}=H\sum_{j=0}^{n-1} \ee^{(n-1-j)H \mathcal{L}} \mathcal{B}_j (\hat{e}_j)\hat{e}_j + \sum_{j=0}^{n-1} \ee^{jH \mathcal{L}}\Big(\breve{e}_{n-j} + H\sum_{i=2}^{s} B_{n-1-j,i}\hat{\varepsilon}_{n-1-j,i}+\hat{\varepsilon}_{n-j}\Big)
\end{equation}
Since the ExpRK method has global order $p$, we have the local
error $\breve{e}_{n-j}=\mathcal{O}(H^{p+1})$, and from
\eqref{fastError} we have
$\hat{\varepsilon}_{n-1-j,i}=\mathcal{O}(Hh^q)$,
$\hat{\varepsilon}_{n-j}=\mathcal{O}(Hh^r)$.
Using \eqref{eq:bound1} we derive from \eqref{eq3.34} that
\begin{equation} \label{eq3.35}
  \|\hat{e}_{n}\| \leq H\sum_{j=0}^{n-1}C \|\hat{e}_j \| +
  \sum_{j=0}^{n-1} C\big(c_1 H^{p+1} +c_2 H^2 h^{q} + c_3 H h^r \big).
\end{equation}
An application of a discrete Gronwall lemma to \eqref{eq3.35} results
in the bound \eqref{eq3.31}.
\end{proof}
\begin{remark}
Since $h=H/m$, Theorem~\ref{theorem2} implies that for a MERK method
\eqref{eq13nc}--\eqref{eq:hatDni} to converge with order $p$, the inner
ODE solvers for \eqref{eq13nc} and \eqref{eq13nd} must have orders $q
\ge p-1$ and $r\ge p$, respectively.
\end{remark}
\section{Derivation of MERK methods}
\label{section4}
Based on the theory presented in Section~\ref{section3}, we now derive
MERK schemes up to order 5, relying heavily on ExpRK schemes that fit
the assumption of Theorem~\ref{theorem1}.  As we are interested in
problems with significant time scale separation $H\gg h$, we primarily
focus on stiffly-accurate ExpRK schemes.  Since MERK methods involve
linear ODEs \eqref{eq13nc} and \eqref{eq13nd} with a fixed coefficient
matrix $\mathcal{L}$ for the fast portion, they are characterized by
the polynomials defined in \eqref{eq:p_hat} and
\eqref{eq:q_hat}. Therefore, when deriving MERK schemes we display 
only their corresponding polynomials $\hat{p}_{n,i}(\tau)$ and
$\hat{q}_n(\tau)$.
\subsection{Second-order methods}
When searching for stiffly-accurate second-order ExpRK methods, we find the following scheme that uses $s=2$
stages (see \cite[Sect.~5.1]{HO05b}) and satisfies
Theorem~\ref{theorem1}: 
\begin{equation} \label{eq:expRK2}
\begin{aligned}
  U_{n,2}&= u_n + c_2 H \varphi _{1} ( c_2 H\mathcal{L})F(t_n, u_n)  \\
  u_{n+1}& = u_n + H \varphi _{1} ( H\mathcal{L})F(t_n, u_n) + h \tfrac{1}{c_2} \varphi _{2} ( H\mathcal{L})D_{n,2}. 
\end{aligned}
\end{equation}
From this, using the conclusion of Theorem~\ref{theorem1}, we derive
the corresponding family of second-order MERK methods, which we call \texttt{MERK2}:
\begin{equation} \label{eq:MERK2}
\begin{aligned}
\hat{p}_{n,2}(\tau)&=  \mathcal{N}(t_n, \hat{u}_n),   \hspace{2cm} \tau \in [0, c_2 H] \\
\hat{q}_{n}(\tau) &=  \mathcal{N}(t_n, \hat{u}_n)+\tfrac{\tau}{c_2 H} \widehat{D}_{n,2}, \q \  \tau \in [0, H].
\end{aligned}
\end{equation}
Since we do not use this scheme in our numerical experiments, we do
not specify a value for $c_2$.  We note that for these methods, the
fast time scale must be evolved a duration of $(1+c_2)H$ for each slow
time step.
\subsection{Third-order methods}
Also from \cite[Sect.~5.2]{HO05b} we consider the following family of
third-order, three-stage, ExpRK methods that satisfy Theorem~\ref{theorem1}:
\begin{equation} \label{eq:expRK3}
\begin{aligned}
  U_{n,2}&= u_n + c_2 H \varphi _{1} ( c_2 H\mathcal{L})F(t_n, u_n)  \\
  U_{n,3}&= u_n + \tfrac{2}{3}H \varphi _{1} ( \tfrac{2}{3} H\mathcal{L})F(t_n, u_n)+\tfrac{4}{9 c_2}  \varphi_{2} ( \tfrac{2}{3} H\mathcal{L})D_{n,2},  \\
  u_{n+1}& = u_n + H \varphi _{1} ( H\mathcal{L})F(t_n, u_n) + h \tfrac{3}{2} \varphi _{2} ( H\mathcal{L})D_{n,3}. 
\end{aligned}
\end{equation}
From these, we construct the following third-order \texttt{MERK3} scheme:
\begin{equation} \label{eq:MERK3}
\begin{aligned}
\hat{p}_{n2}(\tau)&=  \mathcal{N}(t_n, \hat{u}_n), \hspace{2cm} \tau \in [0, c_2 H] \\
\hat{p}_{n3}(\tau)&=  \mathcal{N}(t_n, \hat{u}_n)+ \tfrac{\tau}{c_2 H}\widehat{D}_{n,2}, \q \tau \in [0, \tfrac{2}{3} H] \\
\hat{q}_{n}(\tau) &= \mathcal{N}(t_n, \hat{u}_n) + \tfrac{3\tau}{2 H} \widehat{D}_{n,3},  \q  \ \tau \in [0, H].
\end{aligned}
\end{equation}
In our numerical experiments with this scheme, we choose
$c_2=\tfrac{1}{2}$.  Hence, the fast time scale must be evolved a
duration of $\frac{13}{6}H$ for each slow time step.
\subsection{Fourth-order methods}
To the best of our knowledge, the only 5 stage, stiffly-accurate
ExpRK method of order four was given in \cite[Sect. 5.3]{HO05b}.
However, this scheme does not satisfy Theorem~\ref{theorem1} due to
the coefficient
\[
  a_{52}(H\mathcal{L})=\tfrac{1}{2}\varphi_{2}(c_5 H\mathcal{L})-\varphi_{3}(c_4 H\mathcal{L})+ \tfrac{1}{4}\varphi
  _{2}(c_4 H\mathcal{L})-\tfrac{1}{2}\varphi_{3}(c_5 H\mathcal{L}),
\]
which is not a linear combination of $\{\varphi_{k}(c_5
H\mathcal{L})\}_{k=1}^5$.  Therefore, we cannot use it to derive a
fourth-order MERK scheme.  However, in a very recent submitted paper
\cite{Luan19}, we have derived a family of fourth-order, 6-stage,
stiffly-accurate ExpRK methods (named \texttt{expRK4s6}), that
additionally fulfill Theorem~\ref{theorem1}:
\begin{equation}\label{eq:expRK4}
  \begin{aligned}
    U_{n,2} = u_n &+\varphi_1 (c_2 H\mathcal{L})  c_2 HF(t_n, u_n), \\
    U_{n,k} = u_n &+ \varphi_1 (c_k H\mathcal{L}) c_k HF(t_n, u_n)+ \varphi_2 (c_k H\mathcal{L}) \tfrac{c^2_k}{c_2} H D_{n,2},  \q  \hspace{1cm} k=3, 4 \\
    U_{n,j} = u_n &+ \varphi_1 (c_j H\mathcal{L})c_j h F(t_n, u_n)+ \varphi_{2} (c_j H\mathcal{L}) \tfrac{c^2_j}{c_3-c_4} H \big(\tfrac{-c_4}{c_3}D_{n,3} +\tfrac{c_3}{c_4}D_{n,4}\big)\\
    &+  \varphi_{3} (c_j H\mathcal{L}) \tfrac{2c^3_j}{c_3-c_4} H \big(\tfrac{1}{c_3}D_{n,3} -\tfrac{1}{c_4}D_{n,4}\big), \q \q   \hspace{2.2cm}  j=5,6 \\
    u_{n+1} = u_n &+ \varphi_1 (H\mathcal{L}) h F(t_n, u_n)+ \varphi_{2} (H\mathcal{L}) \tfrac{1}{c_5-c_6} H \big(\tfrac{-c_6}{c_5}D_{n,5} +\tfrac{c_5}{c_6}D_{n,6}\big) \\
    &+\varphi_{3} (H\mathcal{L}) \tfrac{2}{c_5-c_6} H \big(\tfrac{1}{c_5}D_{n,5} -\tfrac{1}{c_6}D_{n,6}\big). 
\end{aligned}
\end{equation} 
Since the pairs of internal stages $\{U_{n,3}, U_{n,4}\}$ and $\{U_{n,5}, U_{n,6}\}$
are independent of one other (they can be computed simultaneously)
and have the same format, this scheme behaves like a 4-stage method.  
Hence, instead of using 6 polynomials we need only 4 to derive the
following family of fourth-order MERK schemes, which we call \texttt{MERK4}:  
\begin{equation} \label{eq:MERK4}
\begin{aligned}
  \hat{p}_{n,2}(\tau) &= \mathcal{N}(t_n, \hat{u}_n),   \q \hspace{6.05cm}  \tau \in [0, c_2 H] \\
  \hat{p}_{n,3}(\tau) &= \hat{p}_{n,4}(\tau) = \mathcal{N}(t_n, \hat{u}_n)+ \tfrac{\tau}{c_2 H}\widehat{D}_{n,2},  \q \hspace{2.8cm} \tau \in [0, c_3 H] \\
  \hat{p}_{n,5}(\tau) &= \hat{p}_{n,6}(\tau) = \mathcal{N}(t_n, \hat{u}_n)+ \tfrac{\tau}{H}\big(\tfrac{-c_4}{c_3(c_3 - c_4)}\widehat{D}_{n,3} + \tfrac{c_3}{c_4(c_3 - c_4)}\widehat{D}_{n,4} \big)\\
  & + \tfrac{\tau^2}{H^2}\big(\tfrac{1}{c_3(c_3-c_4)}\widehat{D}_{n,3} - \tfrac{1}{c_4(c_3 - c_4)}\widehat{D}_{n,4} \big), \q  \hspace{2.4cm} \tau \in [0, c_5 H] \\
  \hat{q}_{n}(\tau) & = \mathcal{N}(t_n, \hat{u}_n) + \tfrac{\tau}{H}\big(\tfrac{-c_6}{c_5(c_5 - c_6)}\widehat{D}_{n,5} + \tfrac{c_5}{c_6(c_5 - c_6)}\widehat{D}_{n,6} \big) \\
  &+ \tfrac{\tau^2}{H^2}\big(\tfrac{1}{c_5(c_5-c_6)}\widehat{D}_{n,5} - \tfrac{1}{c_6(c_5 - c_6)}\widehat{D}_{n,6} \big), \q  \hspace{2.4cm}  \tau \in [0,  H].
\end{aligned}
\end{equation}
For our numerical experiments, we choose the coefficients $c_2=c_3=\tfrac{1}{2}$,
$c_4=c_6=\tfrac{1}{3}$, and $c_5=\tfrac{5}{6}$.  With this choice,
we may then solve the linear ODE \eqref{eq13nc} using the polynomial
$\hat{p}_{n,3}(\tau)$ on $[0, c_3 H]$ to get both $\widehat{U}_{n,3}
\approx U_{n,3}=v_{n,3}(c_3 H)$ and $\widehat{U}_{n,4} \approx
U_{n,4}$ (since $c_4<c_3$) without solving an additional fast
differential equation on $[0, c_4 H]$.
Similarly, we may solve the linear ODE \eqref{eq13nc} with the
polynomial $\hat{p}_{n,5}(\tau)$ on $[0, c_5 H]$ to obtain both
$\widehat{U}_{n,5} \approx U_{n,5}$ and $\widehat{U}_{n,6} \approx
U_{n,6}$.  As a result, the fast time scale must only be evolved for a
total duration of $\frac{17}{6}H$ for each slow time step.
\subsection{Fifth-order methods}
Simiar to fourth-order ExpRK methods, there are no
stiffly-accurate fifth-order methods available in the literature
that fulfill Theorem~\ref{theorem1}. In particular, the only existing
fifth-order scheme (\texttt{expRK5s8}, that requires 8 stages) was
constructed in \cite{LO14b}. However, its coefficients
$a_{75}(H\mathcal{L})$, $a_{76}(H\mathcal{L})$, $a_{85}(H\mathcal{L})$,
$a_{86}(H\mathcal{L})$ and $a_{87}(H\mathcal{L})$  involve several
different linear combinations of $\varphi_k (c_i H\mathcal{L})$ with
different scalings $c_6, c_7, c_8$, and may not be used to create a
MERK method.  Again, in \cite{Luan19}, we have constructed a new
family of efficient, fifth-order, 10-stage, stiffly-accurate ExpRK
methods (called \texttt{expRK5s10}) that fulfills Theorem~\ref{theorem1}:
\begin{subequations}\label{eq:expRK5}
  \begin{equation}\label{eq:expRK5a}
  \begin{aligned}
    U_{n,2} = u_n &+ \varphi_1 (c_2 H\mathcal{L}) c_2 HF(t_n,u_n), \\
    U_{n,k} = u_n &+ \varphi_1 (c_k H\mathcal{L}) c_k HF(t_n, u_n) + \varphi_2 (c_k H\mathcal{L}) \tfrac{c^2_k}{c_2} H D_{n,2},  \hspace{0.9cm} k=3,4 \\
    U_{n,j} = u_n &+ \varphi_1 (c_j H\mathcal{L})c_j HF(t_n, u_n)+ \varphi_{2} (c_j H\mathcal{L}) c^2_j  H \big(\alpha_3 D_{n,3} +\alpha_4 D_{n,4}\big)\\
    &+  \varphi_{3} (c_j H\mathcal{L}) c^3_j H \big(\beta_3 D_{n,3} -\beta_4 D_{n,4}\big), \q \hspace{2.5cm} j=5,6,7 \\
  \end{aligned}
\end{equation} 
  \begin{equation}\label{eq:expRK5b}
  \begin{aligned}
    U_{n,m} = u_n &+ \varphi_1 (c_m H\mathcal{L})c_m H F(t_n, u_n)\\
    &+ \varphi_{2} (c_m H\mathcal{L}) c^2_m  H \big(\alpha_5 D_{n,5} +\alpha_6 D_{n,6}+\alpha_7 D_{n,7} \big)\\
    &+ \varphi_{3} (c_m H\mathcal{L}) c^3_m H \big(\beta_5 D_{n,5} -\beta_6 D_{n,6}-\beta_7 D_{n,7}\big) \\
    &+ \varphi_{4} (c_m H\mathcal{L}) c^4_m H \big(\gamma_5 D_{n,5} +\gamma_6 D_{n,6}+\gamma_7 D_{n,7}\big), \hspace{1.1cm} m=8,9,10 \\
    u_{n+1} = u_n &+ \varphi_1 (H\mathcal{L}) H F(t_n, u_n)+ \varphi_{2} (H\mathcal{L}) H \big(\alpha_8 D_{n,8} + \alpha_9 D_{n,9} +\alpha_{10} D_{n,10} \big) \\
    &-\varphi_{3} (H\mathcal{L})  H \big(\beta_8 D_{n,8} + \beta_9 D_{n,9} +\beta_{10} D_{n,10} \big)\\
    &+\varphi_{4} (H\mathcal{L})  H \big(\gamma_8 D_{n,8} + \gamma_9 D_{n,9} +\gamma_{10} D_{n,10} \big)
  \end{aligned}
\end{equation} 
with coefficients given by
\begin{equation}\label{eq:coefficients}
  \begin{aligned}
    \alpha_3 &= \tfrac{c_4}{c_3 (c_4-c_3)},  \ \alpha_4=\tfrac{c_3}{c_4 (c_3-c_4)},\\
    \alpha_5 &= \tfrac{c_6 c_7}{c_5 (c_5-c_6)(c_5 - c_7)},\  \alpha_6=\tfrac{c_5 c_7}{c_6 (c_6-c_5)(c_6 - c_7)}, \  \alpha_7=\tfrac{c_5 c_6}{c_7 (c_7-c_5)(c_7 - c_6)},\\
    \alpha_8 &= \tfrac{c_9 c_{10}}{c_8 (c_8-c_9)(c_8 - c_{10})},\   \alpha_9=\tfrac{c_8 c_{10}}{c_9 (c_9-c_8)(c_9 - c_{10})},\ \alpha_{10}=\tfrac{c_8 c_{9}}{c_{10} (c_{10}-c_8)(c_{10} - c_{9})}   \\
    \beta_3 &= \tfrac{2}{c_3 (c_3-c_4)},  \beta_4=\tfrac{2}{c_4 (c_3-c_4)}, \\
    \beta_5 &= \tfrac{2(c_6+ c_7)}{c_5 (c_5-c_6)(c_5 - c_7)},\  \beta_6=\tfrac{2(c_5 +c_7)}{c_6 (c_6-c_5)(c_6 - c_7)}, \  \beta_7=\tfrac{2(c_5+ c_6)}{c_7 (c_7-c_5)(c_7 - c_6)},\\
    \beta_8 &= \tfrac{2(c_9 +c_{10})}{c_8 (c_8-c_9)(c_8 - c_{10})},\   \beta_9=\tfrac{2(c_8+ c_{10})}{c_9 (c_9-c_8)(c_9 - c_{10})},\ \beta_{10}=\tfrac{2(c_8 +c_{9})}{c_{10} (c_{10}-c_8)(c_{10} - c_{9})}\\
    \gamma_5 &= \tfrac{6}{c_5 (c_5-c_6)(c_5 - c_7)},\  \gamma_6=\tfrac{6}{c_6 (c_6-c_5)(c_6 - c_7)}, \  \gamma_7=\tfrac{6}{c_7 (c_7-c_5)(c_7 - c_6)},\\
    \gamma_8 &= \tfrac{6}{c_8 (c_8-c_9)(c_8 - c_{10})},\   \gamma_9=\tfrac{6}{c_9 (c_9-c_8)(c_9 - c_{10})},\ \gamma_{10}=\tfrac{6}{c_{10} (c_{10}-c_8)(c_{10} - c_{9})}.
  \end{aligned}
\end{equation} 
\end{subequations}
Although this scheme has 10 stages, again its structure facilitates an
efficient implementation.  Specifically, we note that there are
multiple stages $U_{n,i}$ which share the same format (same matrix
functions with different inputs $c_i$), and are independent of one another
(namely, $\{U_{n,3}, U_{n,4}\}$, $\{U_{n,5}, U_{n,6}, U_{n,7}\}$, and
$\{U_{n,8}, U_{n,9}, U_{n,10}\}$).  These groups of stages can again
be computed simultaneously, allowing the scheme to behave like a
5-stage method.  We therefore propose the corresponding fifth-order
MERK methods that use only 5 polynomials, which we name \texttt{MERK5}:
\begin{equation} \label{eq:MERK5}
  \begin{aligned}
    \hat{p}_{n,2}(\tau)&=  \mathcal{N}(t_n, \hat{u}_n),   \hspace{6.75cm} \tau \in [0, c_2 H] \\
    \hat{p}_{n,3}(\tau)&=\hat{p}_{n,4}(\tau)=  \mathcal{N}(t_n, \hat{u}_n)+ \tfrac{\tau}{c_2 H}\widehat{D}_{n,2},  \hspace{3.5cm}   \tau \in [0, c_3 H] \\
    \hat{p}_{n,5}(\tau)&=\hat{p}_{n,6}(\tau)=\hat{p}_{n,7}(\tau)=  \mathcal{N}(t_n, \hat{u}_n)+ \tfrac{\tau}{H}\big(\alpha_3 \widehat{D}_{n,3} + \alpha_4 \widehat{D}_{n,4} \big)\\
    &  \hspace{3cm}+ \tfrac{\tau^2}{2 H^2}\big(\beta_3 \widehat{D}_{n,3} -\beta_3 \widehat{D}_{n,4} \big), \hspace{1.8cm}  \tau \in [0, c_5 H] \\
    \hat{p}_{n,8}(\tau)&=\hat{p}_{n,9}(\tau)=\hat{p}_{n,10}(\tau)=  \mathcal{N}(t_n, \hat{u}_n)+ \tfrac{\tau}{H}\big(\alpha_5 \widehat{D}_{n,5} +\alpha_6 \widehat{D}_{n,6} +\alpha_7 \widehat{D}_{n,7} \big)\\
    & \hspace{3cm} - \tfrac{\tau^2}{2 H^2}\big(\beta_5 \widehat{D}_{n,5} + \beta_6 \widehat{D}_{n,6} +\beta_7 \widehat{D}_{n,7} \big) \\
    & \hspace{3cm} + \tfrac{\tau^3}{6 H^3}\big(\gamma_5 \widehat{D}_{n,5} +\gamma_6 \widehat{D}_{n,6}+\gamma_7 \widehat{D}_{n,7}\big), \hspace{0.4cm}   \tau \in [0, c_8 H] \\
    \hat{q}_{n}(\tau)&=  \mathcal{N}(t_n, \hat{u}_n) + \tfrac{\tau}{H}(\alpha_8 \widehat{D}_{n,8} + \alpha_9 \widehat{D}_{n,9} +\alpha_{10} \widehat{D}_{n,10})\\
    &\hspace{1.9cm}- \tfrac{\tau^2}{2 H^2} (\beta_8 \widehat{D}_{n,8} + \beta_9 \widehat{D}_{n,9} +\beta_{10} \widehat{D}_{n,10} )\\
    &\hspace{1.9cm}+ \tfrac{\tau^3}{6 H^3}\big(\gamma_8 \widehat{D}_{n,8} + \gamma_9 \widehat{D}_{n,9} +\gamma_{10} \widehat{D}_{n,10} \big), \hspace{1.25cm} \tau \in [0,  H].
  \end{aligned}
\end{equation}
For our numerical experiments, we choose
$c_2=c_3=c_5=c_9=\tfrac{1}{2}$, $c_4=c_6=\tfrac{1}{3}$, $c_7=\tfrac{1}{4}$,
$c_8=\tfrac{7}{10}$, and $c_{10}=\tfrac{2}{3}$.  Again, since $c_4<c_3$,
when solving the fast time-scale problem \eqref{eq13nc} with
polynomial $\hat{p}_{n,3}(\tau)$ on $[0, c_3 H]$ gives
$\widehat{U}_{n,3} \approx U_{n,3}=v_{n,3}(c_3 h)$ and
$\widehat{U}_{n,4} \approx U_{n,4}=v_{n,3}(c_4 h)$. Similarly, since
$c_7<c_6<c_5$, $\widehat{U}_{n,5}, \widehat{U}_{n,6}$, and
$\widehat{U}_{n,7}$ can be obtained by solving a single fast time-scale
problem with polynomial $\hat{p}_{n,5}(\tau)$ on $[0,
c_5 H]$.  Finally since $c_9<c_{10}<c_8$, one can compute
$\widehat{U}_{n,8}$, $\widehat{U}_{n,9}$, and $\widehat{U}_{n,10}$ by
solving a single fast time-scale problem with
polynomial $\hat{p}_{n,8}(\tau)$ on $[0, c_8 H]$.  The sum total of
these solves corresponds to evolving the fast time scale for an overall
duration of $\frac{16}{5}H$ for each slow time step.


\section{Numerical experiments}
\label{sec6}
In this section we present results from a variety of numerical tests
to examine the performance of the proposed {\merkthree}, {\merkfour} and
{\merkfive} methods.  These tests are designed to confirm the 
theoretical convergence rates from Section \ref{sec:analysis}, and
compare efficiency against the Multirate Infinitesimal Step method
\texttt{MIS-KW3}, which uses a similar approach of evolving the fast
component using modified systems of differential equations
\cite{knothwolke98,Wensch2009,Schlegel2009,Schlegel2012b}.  Unless
otherwise noted, we run these methods with inner explicit Runge-Kutta
ODE solvers of the same order of convergence as the MERK method, $p$:
\begin{itemize}
\item Third order {\mis} uses the \texttt{Knoth-Wolke-ERK} inner method \cite{knothwolke98};
\item Third order {\merkthree} uses the \texttt{ERK-3-3} inner method,
    $\begin{array}{c|ccc}
      0 &  &  & \\
      1/2 & 1/2 & & \\
      1 & -1 & 2 & \\
      \hline
        & 1/6 & 2/3 & 1/6
    \end{array}$;
\item Fourth order {\merkfour} uses the \texttt{ERK-4-4} inner method
  \cite[Table 1.2, left]{hairer93};
\item Fifth order {\merkfive} uses the \texttt{Cash-Karp-ERK} inner
  method \cite{cashkarp90}.
\end{itemize}
We note that although Theorem \ref{theorem2} guarantees that when
using a MERK method of order $p$, the internal stage solutions
\eqref{eq13nc} can be computed with a solver of order $q = p-1$
and the step solution \eqref{eq13nd} can use a solver of order
$r = p$, for simplicity we have used $r=q=p$ in the majority of our
tests. However, we more closely investigate these inner solver
order requirements in Section \ref{subsec:fast} below.

Not all of our test problems have convenient analytical solutions; for
these tests, we compute a reference solution using an 8th order
explicit or a 12th order implicit Runge-Kutta method with a time step
smaller than the smallest micro time step $h$.  When computing
solution error, we report the maximum absolute error over all time
steps and solution components.  From these, we compute convergence
rates using a linear least-squares fit of the log-error versus
log-macro time step $H$. For each test we present three types of
plots: one convergence plot (error vs $H$) and two efficiency plots.
Generally, efficiency plots present error versus the computational
cost. However in the multirate context, fast and slow function
costs can differ dramatically.  As such, we separately consider
efficiency using total function calls and slow function calls.  Since
the dominant number of total calls are from the fast function, the
``total'' plots represent the method efficiency for simulations with
comparable fast/slow function cost, whereas the ``slow-only'' plots
represent the method efficiency for simulations in which the slow
function calls are significantly more expensive (as explained in
Section \ref{section1} as our original motivation for multirate methods).
Individual applications will obviously lie somewhere between these
extremes, but we assume that they are typically closer to the
``slow-only'' results.

Applications scientists traditionally use multirate solvers for one
of two reasons.  The first category are concerned with simulations of
stiff systems, but where they choose to use a subcycled explicit
method instead an implicit one for the stiff portion of the problem.
Generally, these applications are primarily concerned with selecting
$h$ to satisfy stability of the fast time scale (instead of accuracy).
The second category consider simulations wherein it is essential to
capture the coupling between the slow and fast times scales
accurately, since temporal errors at the fast time scale can
significantly deteriorate the slow time scale solution;
here $h$ is chosen based on accuracy considerations. We therefore
separately explore test problems in both of these categories in the
Sections \ref{subsec:category1} and \ref{subsec:category2} below.

To facilitate reproducibility of the results in this section, we have
provided an open-source MATLAB implementation of the {\merkthree},
{\merkfour}, {\merkfive} and {\mis} methods, along with scripts to
perform all tests from this section \cite{MERK_repo}.

\subsection{Category I}
\label{subsec:category1}
As this category of problems is concerned with stability at the fast
time scale, we choose a fixed, linearly stable micro time step $h$,
and vary the macro time step $H$ (and similarly, $m=H/h$).  To this
end, we focus on two stiff applications: a reaction diffusion problem
and the brusselator problem.


\subsubsection{Reaction Diffusion}
\label{subsubsec:reaction_diffusion}
We consider a reaction diffusion problem with a traveling wave
solution similar to the one considered by Savcenco et
al.~\cite{savcenco},
\begin{align*}
  &u_t = \frac{1}{100} u_{xx} + u^2(1 - u), \qquad 0<x<5, \quad 0<t\le 3,\\
  &u_x(0,t) = u_x(5,t) = 0 , \qquad u(x,0) = (1+ e^{\lambda(x-1)})^{-1},
\end{align*}
where $\lambda = 5 \sqrt{2}$.
We discretize in space using a second order accurate central finite
difference scheme using $1000$ spatial points. This gives
us a system for which we take $\mathcal{L}$ and $\mathcal{N}(t,u(t))$
to be the discretized versions of $\frac{1}{100} u_{xx}$ and $u^2(1 -
u)$ respectively. The micro time step is chosen to satisfy the
Courant-Friedrichs-Lewy (CFL) linear stability condition, $h = 10^{-3}$.

In the left of Figure \ref{fig:randd} we plot the method convergence
as $H$ is varied, which shows slighty convergence rates that are
better than predicted for all methods tested.  As this behavior is not
consistently observed for the remaining test problems, we believe that
this is an artifact of this particular test problem.  Here, we compute
the best-fit rates using only the error values larger than
$\sim10^{-13}$, where the error stagnates due to the accuracy of the
reference solution.  

\begin{figure}[h!]
\centering
\begin{subfigure}[b]{0.49\textwidth}
\includegraphics[width =\textwidth]{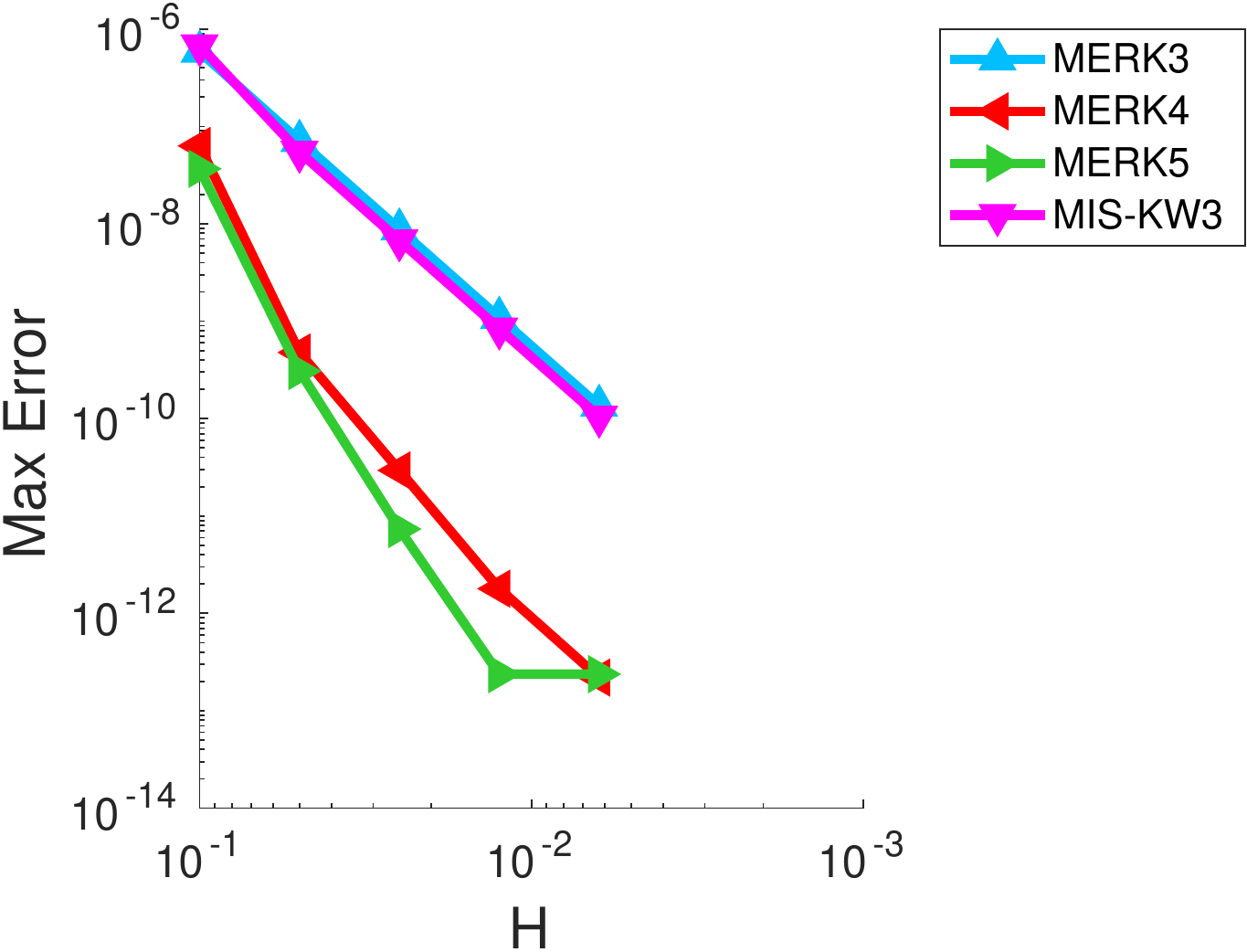}
\end{subfigure}
\begin{subfigure}[b]{0.49\textwidth}
\includegraphics[width = \textwidth]{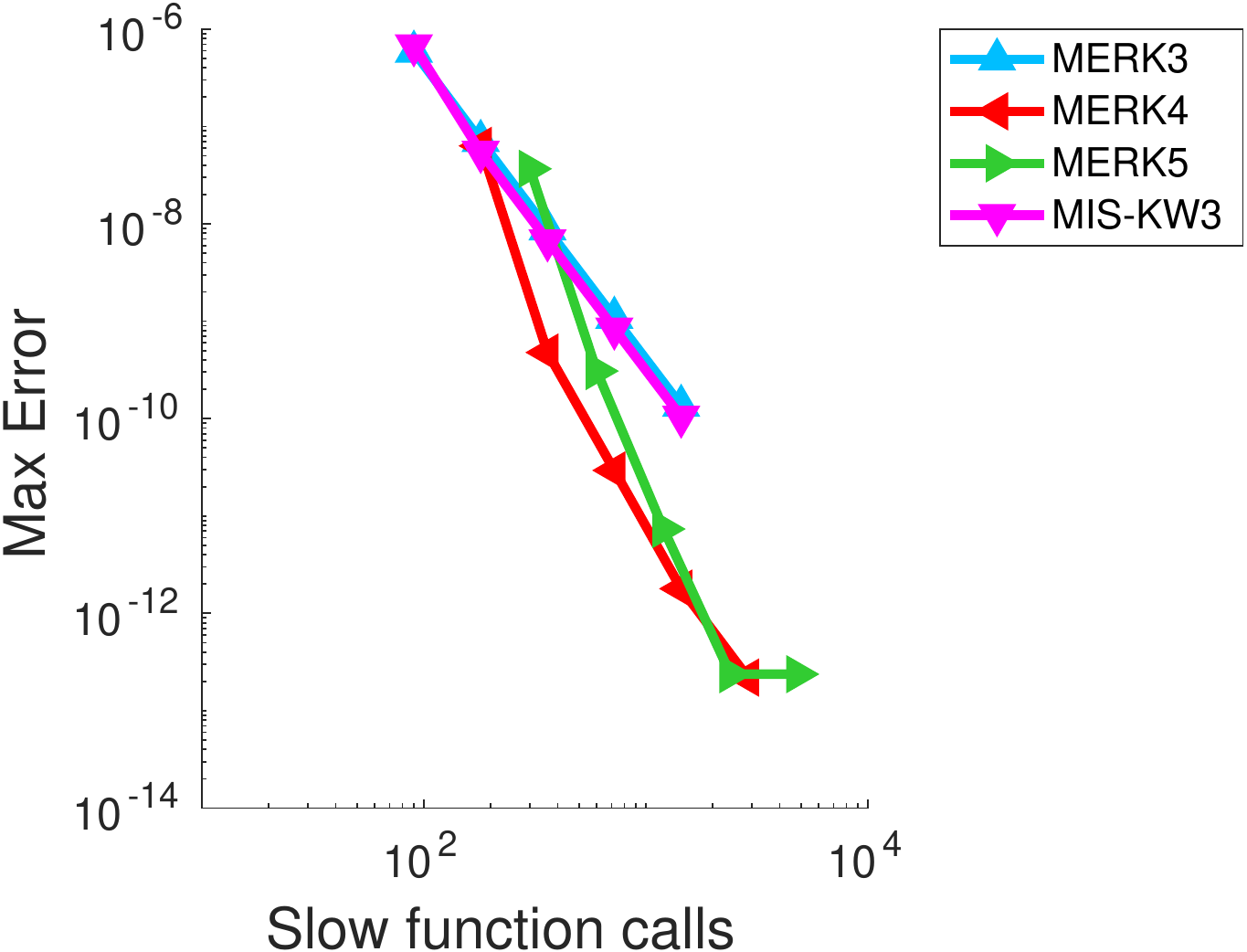}
\end{subfigure}
\caption{\small Reaction diffusion convergence (left) and
  ``slow-only'' efficiency (right). The best fit convergence rates are
  3.03, 4.93, 5.71, 3.20 (\merkthree, \merkfour, \merkfive, and \mis,
  resp.).  The most ``efficient'' methods at a given error are to the
  left of their less efficient counterparts.}
\label{fig:randd}
\end{figure}

The efficiency plots for both test problems in this category are very
similar, so we present the ``slow-only'' efficiency plot for this
problem in the right of Figure \ref{fig:randd}, saving the ``total''
efficiency plot for the next test.  Here, we note that for tolerances
larger than $10^{-7}$, {\merkthree} and {\mis} are the most efficient,
but for tighter tolerances {\merkfour} is the best.  Although
{\merkfive} has a higher rate of convergence, the increased cost per
step causes it to lag behind until it reaches the reference solution
accuracy, where it begins to overtake {\merkfour}.


\subsubsection{Brusselator}
\label{subsubsec:brusselator}
The brusselator is an oscillating chemical reaction problem for which
one of the reaction products acts as a catalyst. It is widely used as
a test for ODE solvers, including IMEX and multirate methods.  We use a
variant of this stiff nonlinear ODE system given by:
\begin{align*}
  \begin{bmatrix}
    u \\ v \\ w
  \end{bmatrix}' &= \begin{bmatrix}
    a - (w + 1)u + u^2v\\ wu - u^2v\\ \frac{b-w}{\epsilon} - uw
  \end{bmatrix},\qquad
  \mathbf{u}(0) = \begin{bmatrix} 1.2 \\ 3.1 \\ 3 \end{bmatrix},
\end{align*}
over the interval $t\in (0,2]$, with parameters $a = 1, b = 3.5$ and
$\frac{1}{\epsilon} = 100$.  We convert this to have the multirate
form (\ref{eq1}) by defining
\begin{align*}
  \mathcal{L} = \begin{bmatrix}
    0 & 0 & 0 \\ 0 & 0 & 0 \\ 0 & 0 & \frac{-1}{\epsilon}
  \end{bmatrix},\hspace{7mm} \mathcal{N}(t,\mathbf{u}(t)) =  \begin{bmatrix}
    a - (w + 1)u + u^2v\\ wu - u^2v\\ \frac{b}{\epsilon} - uw
  \end{bmatrix}.
\end{align*}

In the left of Figure \ref{fig:brus} we plot the error versus $H$, and
list the corresponding best-fit convergence rates.  We observe
that all the tested methods perform slightly worse than their
predicted convergence rates, which we attribute to order reduction due
to the stiffness of the problem; however, the relative convergence
rates of each method compare as expected against one another.

\begin{figure}[h!]
\centering
\begin{subfigure}[b]{0.49\textwidth}
\includegraphics[width = \textwidth]{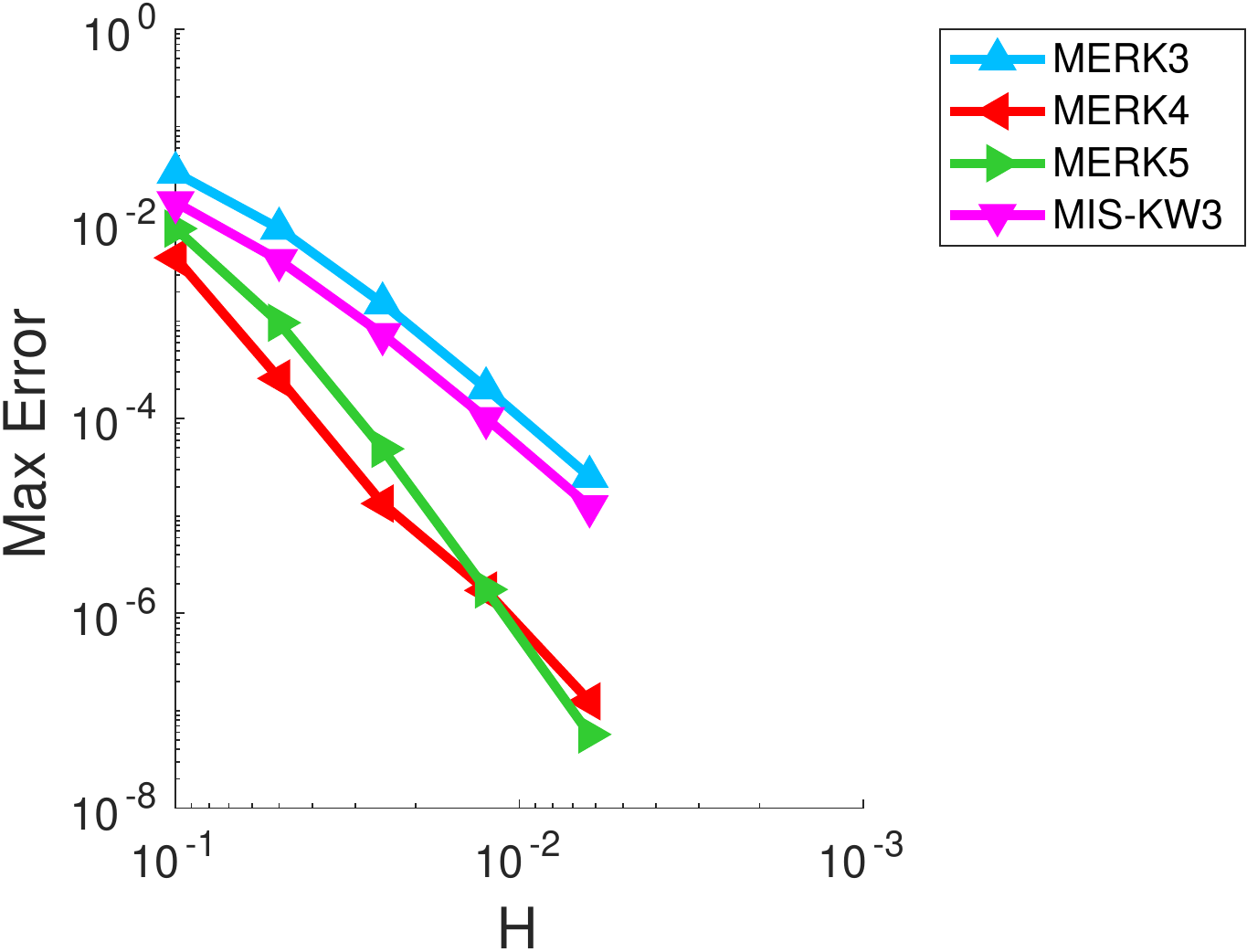}
\end{subfigure}
\begin{subfigure}[b]{0.49\textwidth}
\includegraphics[width = \textwidth]{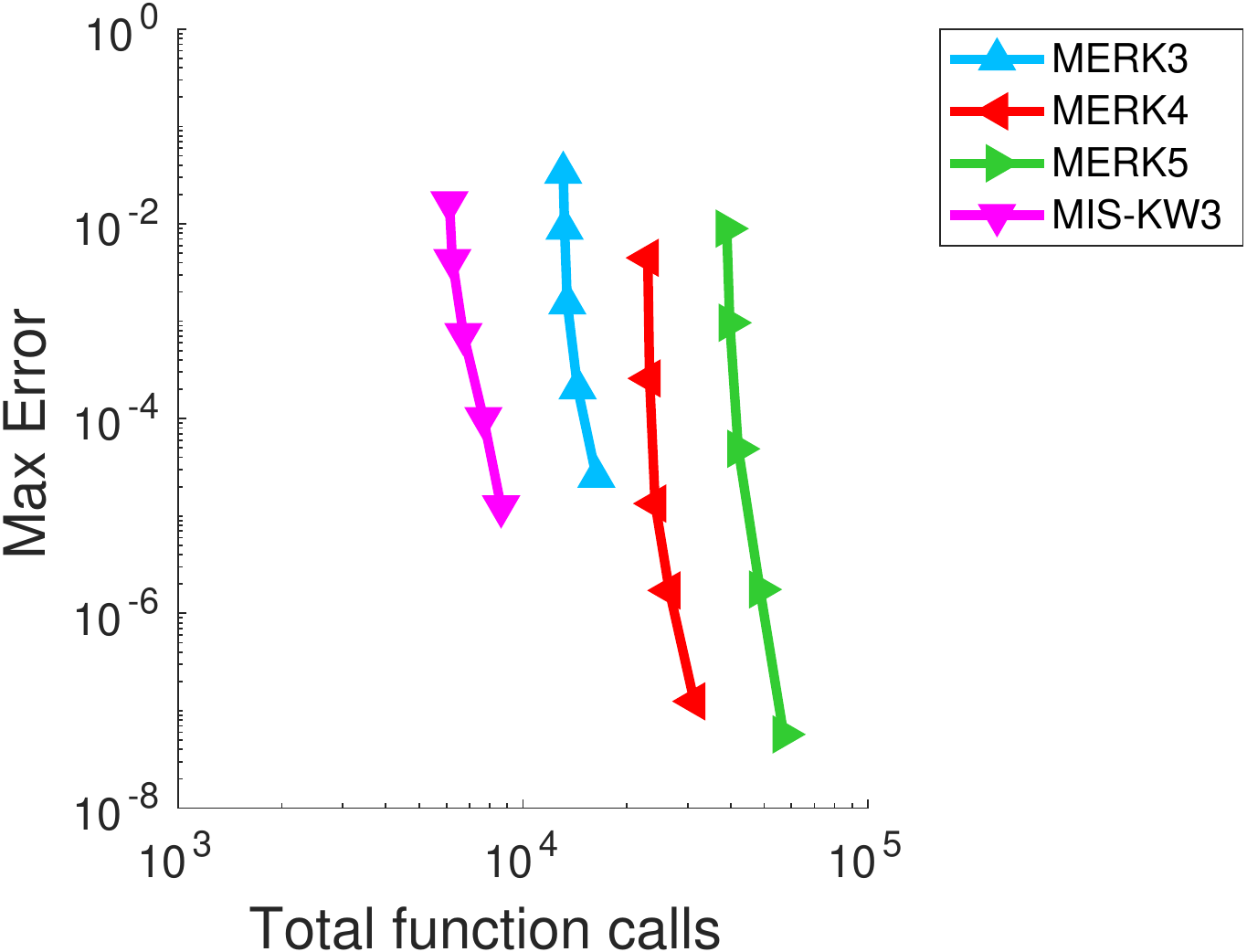}
\end{subfigure}
\caption{\small Brusselator convergence (left) and ``total''
  efficiency (right).  The best fit convergence rates are
  2.62, 3.75, 4.36 and 2.61 (\merkthree, \merkfour, \merkfive, and \mis,
  respectively).  Note the near-vertical lines in the efficiency
  plots, indicating the dominance of ``fast'' function calls in the
  estimate of total cost.}
\label{fig:brus}
\end{figure}

For this test problem, we plot the efficiency based on total function
calls in the right of Figure \ref{fig:brus}.  We note that each curve
is almost vertical since the micro time step $h$ is held constant for
these tests, and is significantly smaller than $H$.  Here, {\mis}
takes the least amount of total function calls since its structure
ensures that it only traverses the time step interval $[t_n,t_n+H]$
once when evaluating the modified ODEs, whereas {\merkthree},
{\merkfour} and {\merkfive} require approximately 2, 3 and 3
traversals, respectively.  We note that although these additional
traversals of the time step interval $[t_n,t_n+H]$ result in
significant increases in the number of fast function calls, the number
of potentially more costly slow function calls for all methods is
equal to the number of slow stages.

\subsection{Category II}
\label{subsec:category2}

Recalling that our second category of multirate applications focuses
on accurately coupling the fast and slow processes, for these test 
problems we choose a fixed time scale separation factor $m$ for each
method/test, and vary $H$ (and proportionally, $h=H/m$).  For this
group of tests we consider a linear multirate problem from Estep et
al.~\cite{estep} for which the fast variables are coupled into the
slow equation (one-directional coupling) and a linear multirate
problem of our own design where both the fast and slow variables are
coupled (bi-directional coupling).  Since the ``optimal'' value of $m$
for each multirate algorithm is problem-dependent, we describe our
approach for determining this $m$ value in Section
\ref{subsubsec:one_directional} below.

\subsubsection{One-directional coupling}
\label{subsubsec:one_directional}
We consider a linear system of ODEs~\cite{estep}:
\begin{align}
  \begin{bmatrix}
    u \\ v \\ w
  \end{bmatrix}' &= \begin{bmatrix}
    0 & -50 & 0 \\ 50 & 0 & 0 \\ 1 & 1 & -1
  \end{bmatrix}\begin{bmatrix}
    u \\ v \\w
  \end{bmatrix}, \label{eq:estep_problem}\qquad
  \mathbf{u}(0) = \begin{bmatrix} 1\\0\\2\end{bmatrix},
\end{align}
over the interval $t\in (0,1]$. This has analytical solution $u(t) =
\cos(50t)$, $v(t) = \sin(50t)$, and $w(t) = \frac{5051}{2501} e^{-t} -
\frac{49}{2501}\cos(50t) + \frac{51}{2501}\sin(50t)$.
We convert this problem to multirate form \eqref{eq1} by decomposing
it as:
\begin{align*}
  \mathcal{L} = \begin{bmatrix}
    0 & -50 & 0 \\ 50 & 0 & 0 \\ 1 & 1 & 0
  \end{bmatrix},\hspace{7mm} \mathcal{N}(t,\mathbf{u}(t)) = \begin{bmatrix}
    0 \\ 0 \\-w
  \end{bmatrix}.
\end{align*}

We first discuss our approach in determining the ``optimal''
time-scale separation factor $m$.  For illustration, we consider
{\merkfour} on this problem; however, we apply this approach to all
methods for both this test and the following bi-directional coupling
test in Section \ref{subsubsec:bidirectional}.  We begin by 
repeatedly solving the problem \eqref{eq:estep_problem} using the
multirate method with different factors $m =
\{5,10,25,50,75,85,100,125\}$. For each value of $m$, we vary $H$ (and
hence $h=H/m$).  We then analyze the resulting ``total'' and
``slow-only'' efficiency plots for each fixed $m$ value, as shown in
Figure \ref{fig:estepprocess_eff}.

\begin{figure}[h!]
\centering
\begin{subfigure}[b]{0.49\textwidth}
\includegraphics[width = \textwidth]{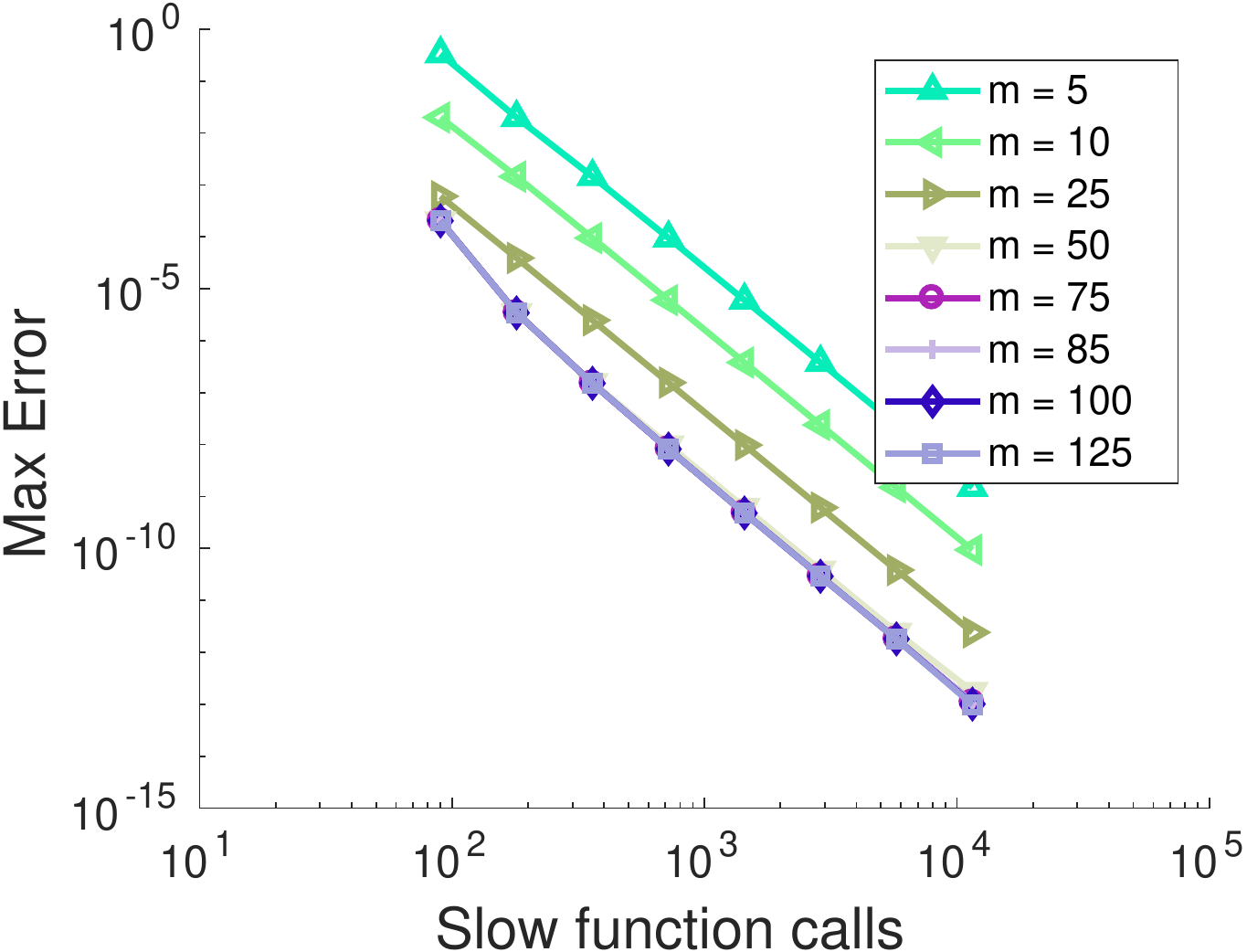}
\caption{}
\label{fig:estepprocess_slo}
\end{subfigure}
\begin{subfigure}[b]{0.49\textwidth}
\includegraphics[width = \textwidth]{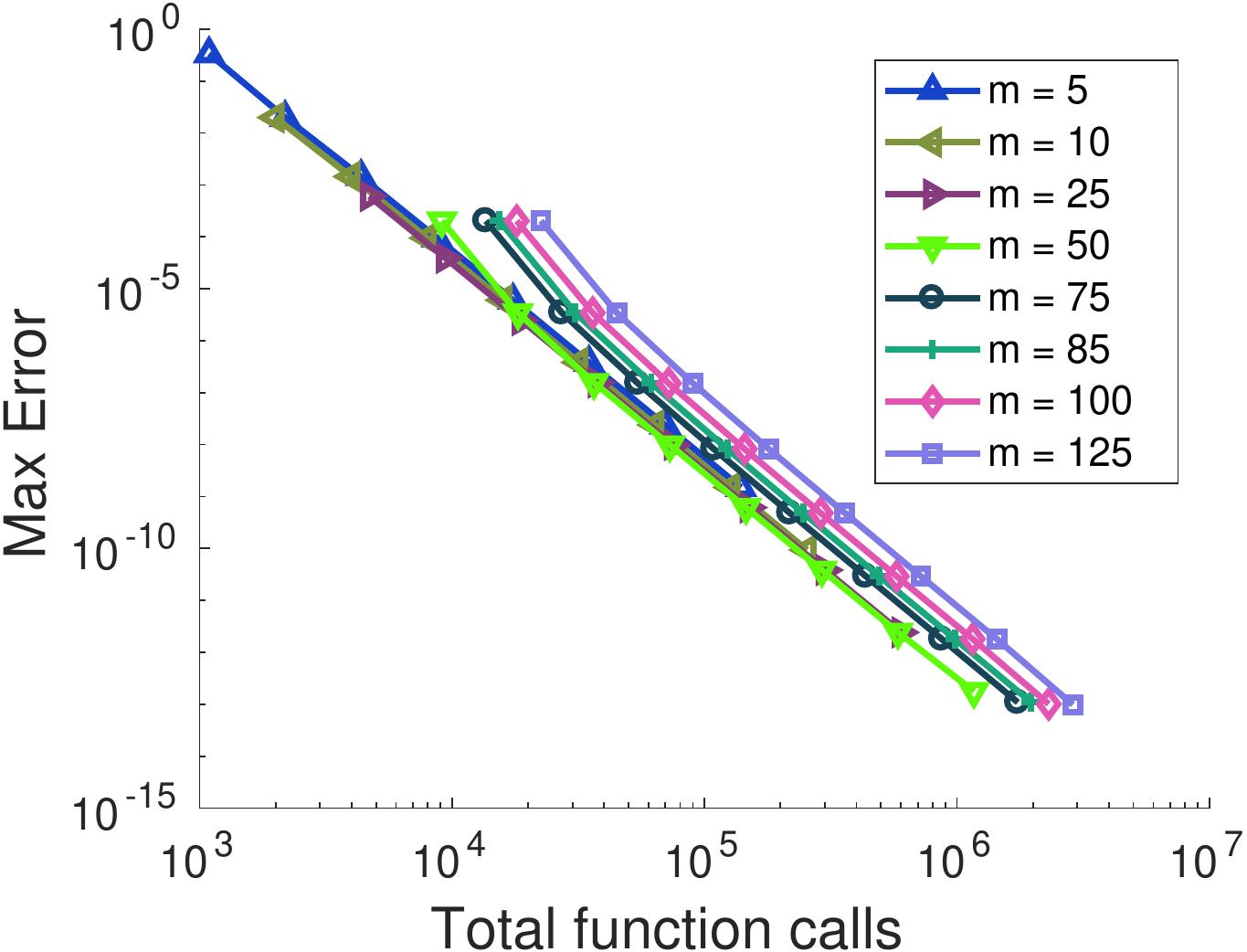}
\caption{}
\label{fig:estepprocess_tot}
\end{subfigure}
\caption{Efficiency plots for {\merkfour} applied to the
  one-directional coupling test, resulting from various $m$ factors.}
\label{fig:estepprocess_eff}
\end{figure}
We first note that both plots show a group of $m$ values with
identical efficiency, along with other less efficient results.  In
Figure \ref{fig:estepprocess_slo}, the more efficient group is
comprised of \emph{larger} $m$ values, whereas in Figure
\ref{fig:estepprocess_tot} the more efficient group has \emph{smaller}
$m$ values.  This is unsurprising, since increases in $m$ for a fixed
$H$ correspond to decreases in $h$, leading to accuracy improvements
at the fast time scale alone.  While this will results in increased
total function calls, the number of slow function calls will remain
fixed.  We therefore define the ``optimal'' $m$ as the value where the
fast and slow solution errors are balanced.  Hence, in Figure
\ref{fig:estepprocess_tot} this corresponds to the largest $m$ that
remains in the more efficient group, and in Figure
\ref{fig:estepprocess_slo} this corresponds to the smallest $m$ that
remains in the more efficient group.  Inspecting both plots in Figure
\ref{fig:estepprocess_eff}, the optimal value for {\merkfour} on this
problem is $m=50$.  Carrying out a similar process for the other
methods on this problem, {\merkthree} has an optimal value of
$m = 75$, {\merkfive} $m = 25$, and {\mis} has an optimal value of
$m = 75$.

Using these $m$ values, In Figure \ref{fig:fasttoslo_conv} we plot the
convergence results for the four methods on this problem, confirming
the analytical orders of convergence, with errors stagnating around
$10^{-13}$ due to accumulation of floating-point roundoff.  While we
find slightly better-than-expected convergence rates for the
MERK methods, and only the expected rate for {\mis}, we do not draw
conclusions regarding this behavior.

\begin{figure}[h!]
\centering
\includegraphics[width = 0.49\textwidth]{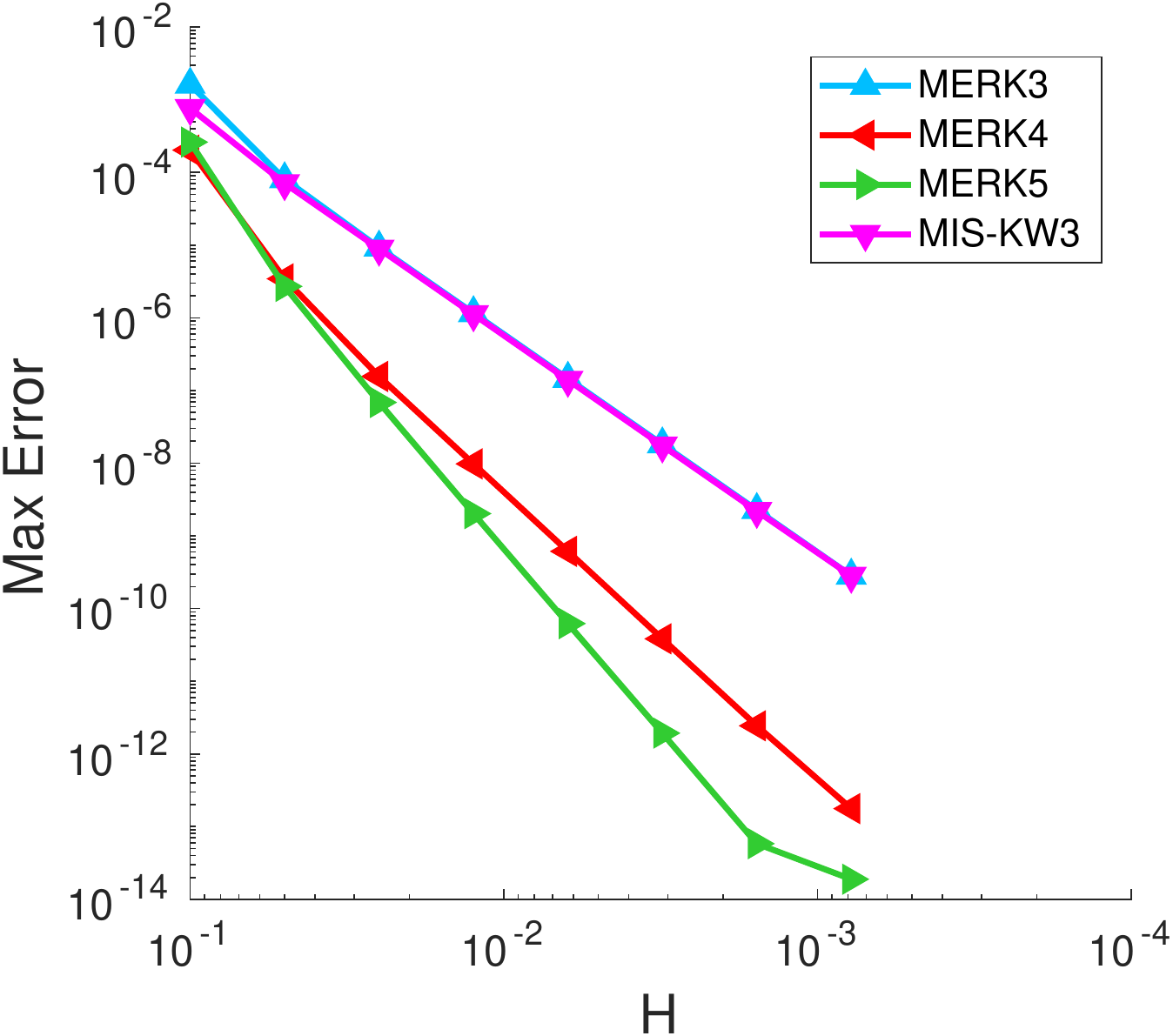}
\caption{One-directional coupling convergence. \small Best fit
  convergence rates are 3.16, 4.28, 5.26 and 3.04 ({\merkthree},
  {\merkfour}, {\merkfive} and {\mis}, respectively).}
\label{fig:fasttoslo_conv}
\end{figure}

Similarly, in Figure \ref{fig:fasttoslo_eff} we plot both the
``total'' and ``slow-only'' efficiency of each method on the
one-directional test problem.  When measuring only slow function
calls, both {\mis} and {\merkthree} tie for errors larger than
$10^{-6}$, {\merkfour} is the most efficient for errors between
$10^{-6}$ and $10^{-12}$ and {\merkfive} is the most efficient at the
tightest error values.  When the fast function calls are given equal
weight as the slow, however, {\mis} is the most efficient at errors
larger than $10^{-8}$, while {\merkfive} is the most efficient at
tighter error thresholds.

\begin{figure}[h!]
\centering
\begin{subfigure}[b]{0.49\textwidth}
\includegraphics[width = \textwidth]{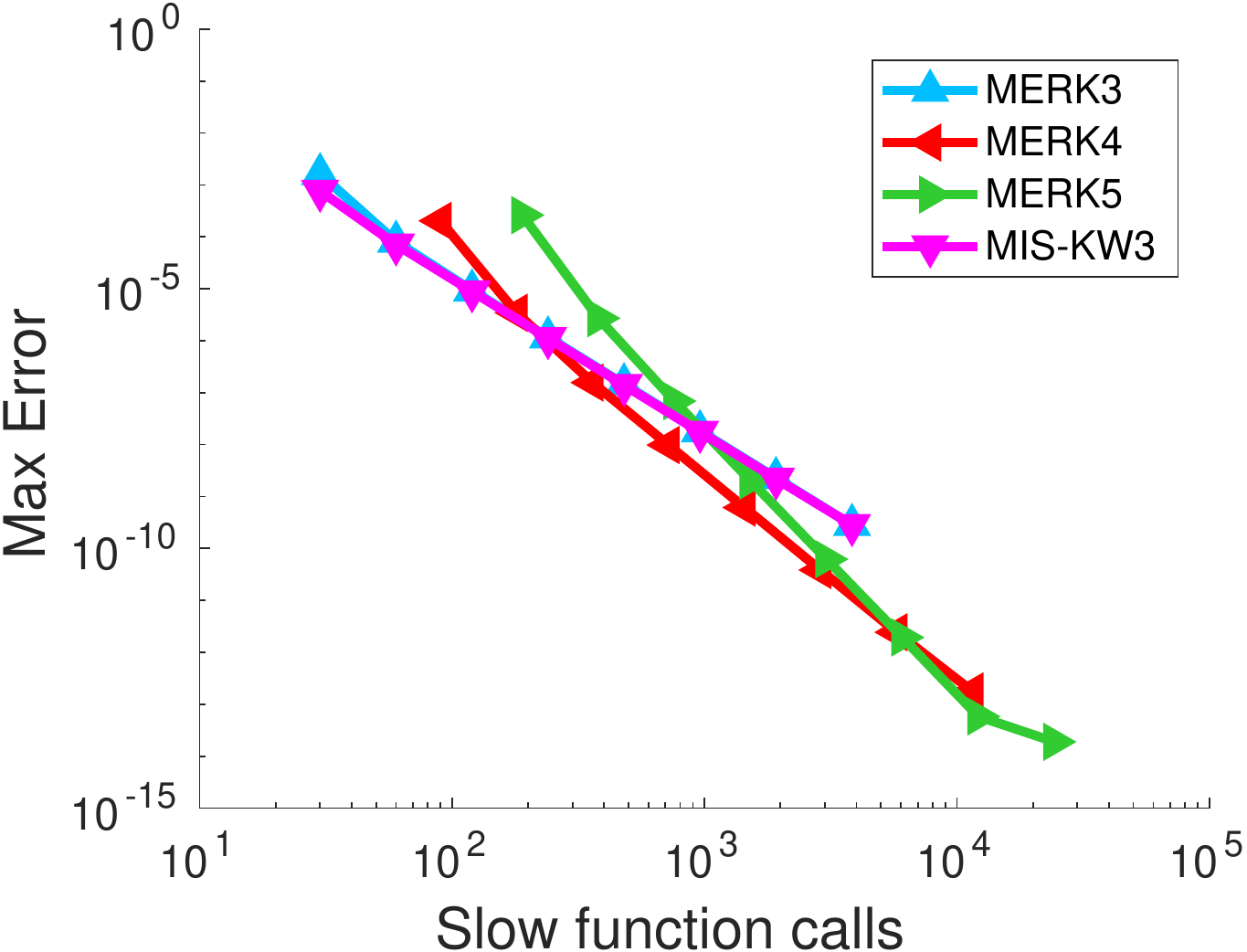}
\caption{}
\label{fig:fasttoslo_slo}
\end{subfigure}
\begin{subfigure}[b]{0.49\textwidth}
\includegraphics[width = \textwidth]{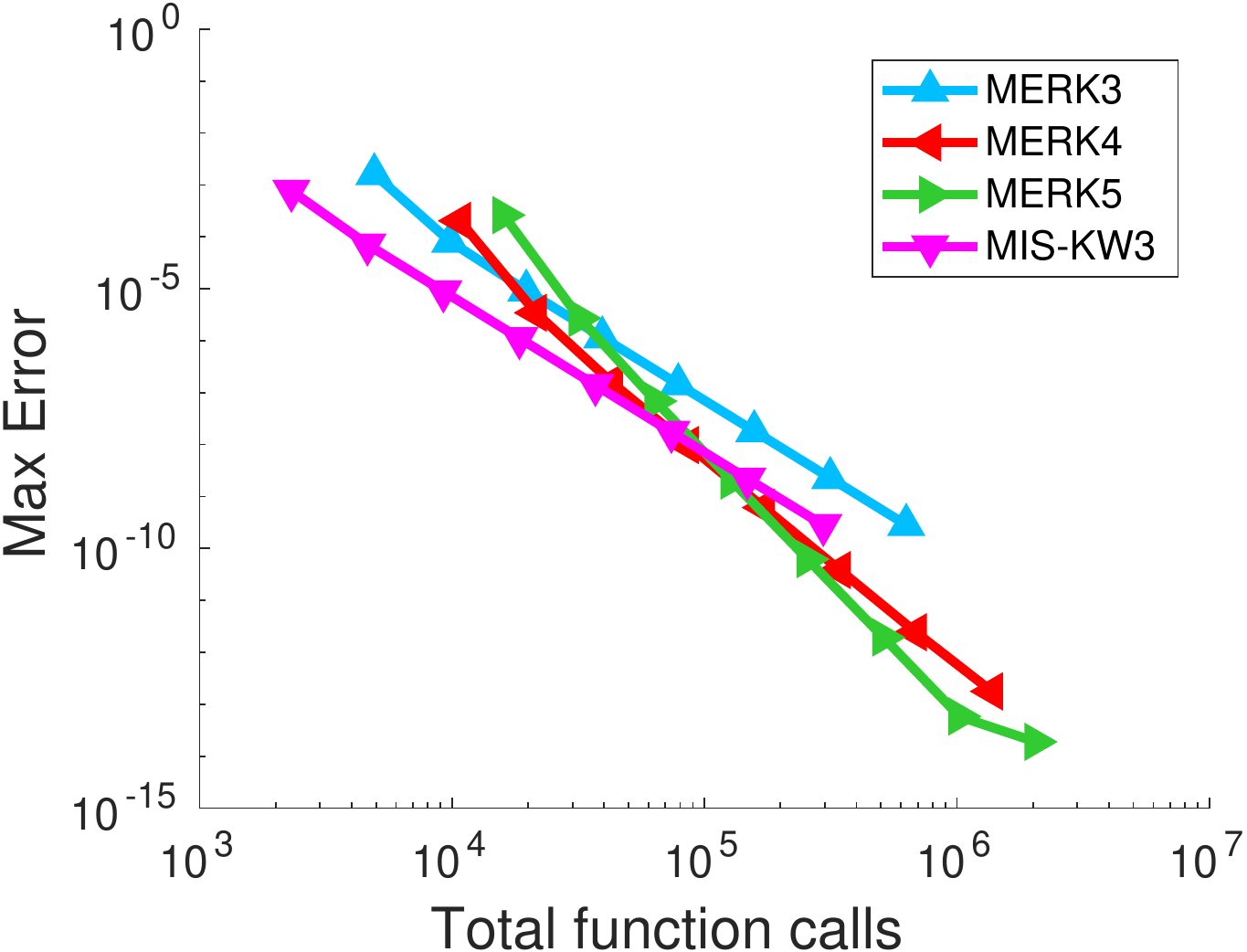}
\caption{}
\label{fig:fasttoslo_tot}
\end{subfigure}
\caption{One-directional coupling efficiency.  The most efficient
  method depends on how ``cost'' is measured, as well as on the
  desired accuracy.}
\label{fig:fasttoslo_eff}
\end{figure}


\subsubsection{Bi-directional coupling}
\label{subsubsec:bidirectional}

Taking inspiration from the preceding one-directional test, we
designed a problem with coupling between both the fast and slow
components to further demonstrate the flexibility and robustness of
MERK methods. To this end, we consider the following test problem
\begin{align}
  \begin{bmatrix}
    u \\ v \\ w
  \end{bmatrix}' &= \begin{bmatrix}
    0 & 100 & 1 \\ -100 & 0 & 0 \\ 1 & 0 & -1
  \end{bmatrix}\begin{bmatrix}
    u \\ v \\w
  \end{bmatrix},\label{eq:bidirectional_pb}\qquad
  \mathbf{u}(0) = \begin{bmatrix} 9001/10001\\
    100000/10001 \\ 1000\end{bmatrix},
\end{align}
over $t\in (0,2]$.  Converting to multirate form \eqref{eq1},
we set $\mathcal{L}$ and $\mathcal{N}(t,\mathbf{u}(t))$ as:
\begin{align*}
  \mathcal{L} = \begin{bmatrix}
    0 & 100 & 0 \\ -100 & 0 & 0 \\ 1 & 0 & 0
  \end{bmatrix},\hspace{7mm} \mathcal{N}(t,\mathbf{u}(t)) = \begin{bmatrix}
    w \\ 0 \\-w
  \end{bmatrix}.
\end{align*}
While this is a linear test problem that may be solved using the
matrix exponential, this solution is difficult to represent in
closed-form, and so we use a reference solution for convenience. 
Using the previously-described approach for determining the optimal
time-scale separation factor $m$ for each method on this problem, we
have $m = 50$ for {\merkthree} and {\merkfour}, $m = 10$ for
{\merkfive} and $m=25$ for {\mis}.

In Figure \ref{fig:fastslocoup_conv} we plot the convergence rates of
each method on this test problem, again confirming the analytical
orders of convergence, with errors stagnating around $10^{-12}$ due to
the reference solution accuracy.

\begin{figure}[h!]
\centering
\includegraphics[width = 0.49\textwidth]{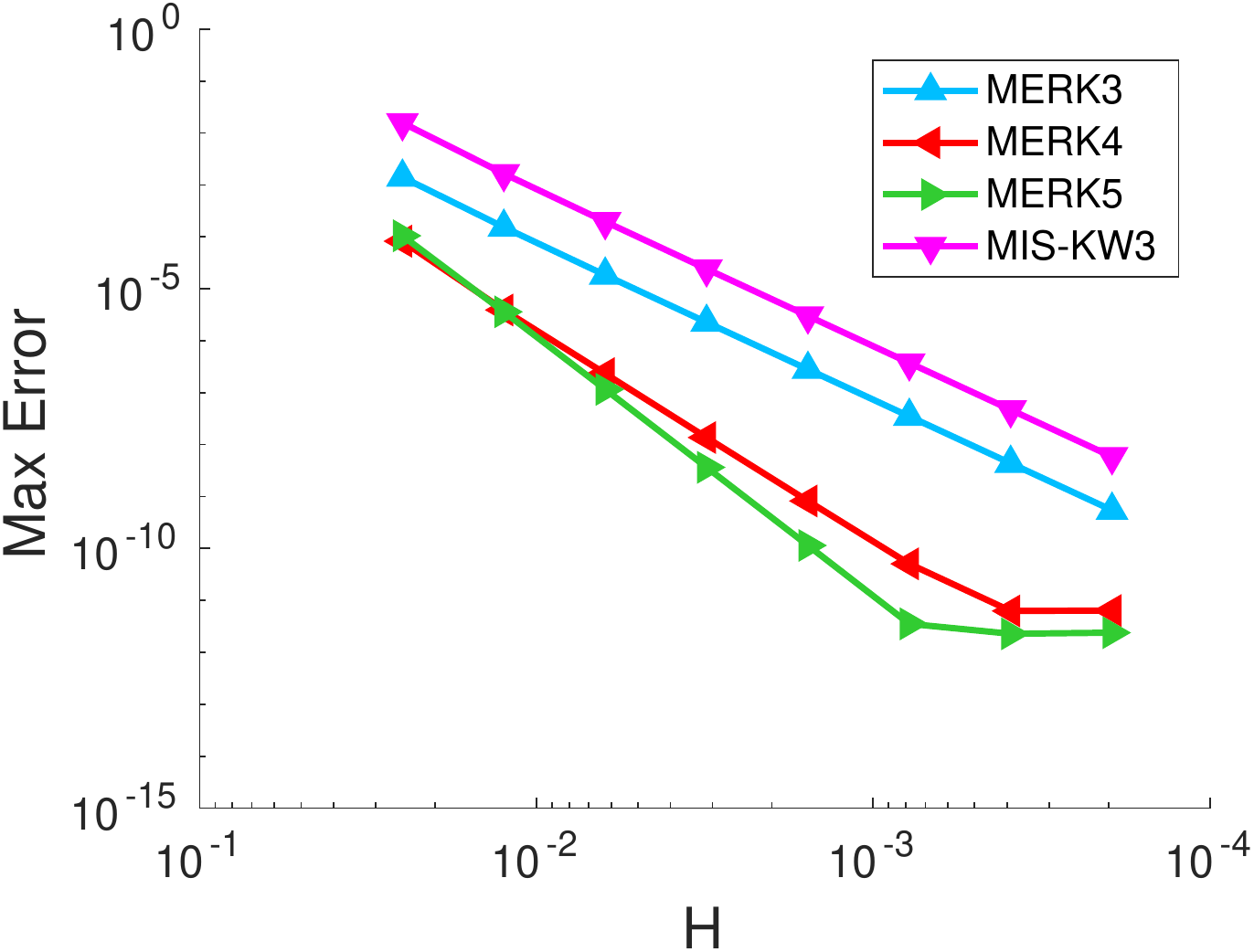}
\caption{\small Bi-directional coupling convergence. Best fit
  convergence rates are 3.03, 3.99, 4.97 and 3.06 ({\merkthree},
  {\merkfour}, {\merkfive}, {\mis}, respectively).}
\label{fig:fastslocoup_conv}
\end{figure}

Similarly, in Figure \ref{fig:fastslocoup_eff} we plot both the
``slow-only'' and ``total'' efficiency plots for this problem.  Here,
when measuring only the slow function calls, the most 
efficient method is {\merkthree} at error thresholds above $10^{-5}$,
and {\merkfive} for smaller error values.  Strikingly, when
considering the total number of function calls, the {\merkfive}
is the most efficient at nearly all error thresholds.  We note,
however, that the optimal time-scale separation factor for {\merkfive}
is $m=10$ for this problem, which results in reduced fast function
calls per slow step, and hence an overal reduction in total function
calls.

\begin{figure}[h!]
\centering
\begin{subfigure}[b]{0.49\textwidth}
\includegraphics[width = \textwidth]{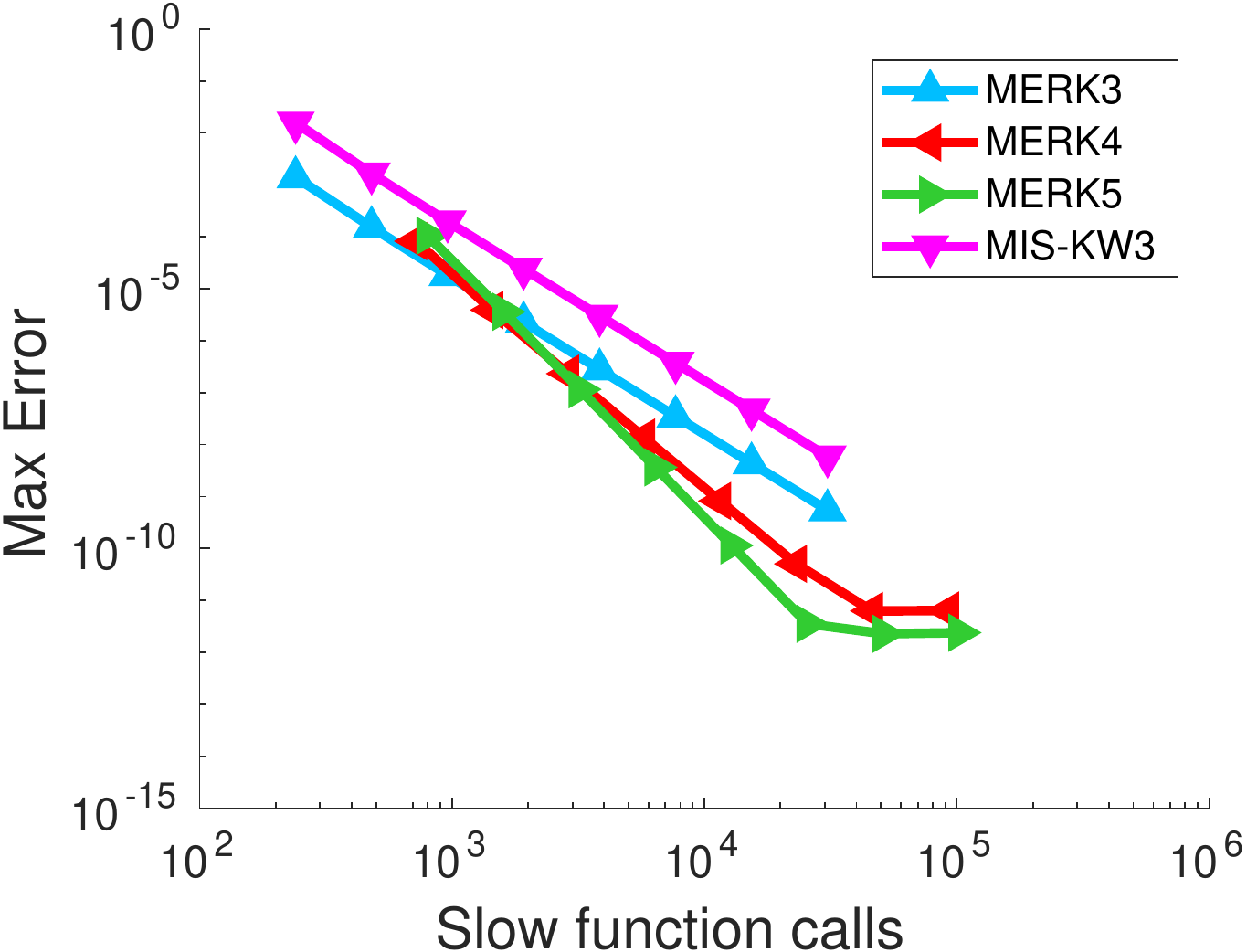}
\caption{}
\label{fig:fastslocoup_slo}
\end{subfigure}
\begin{subfigure}[b]{0.49\textwidth}
\includegraphics[width = \textwidth]{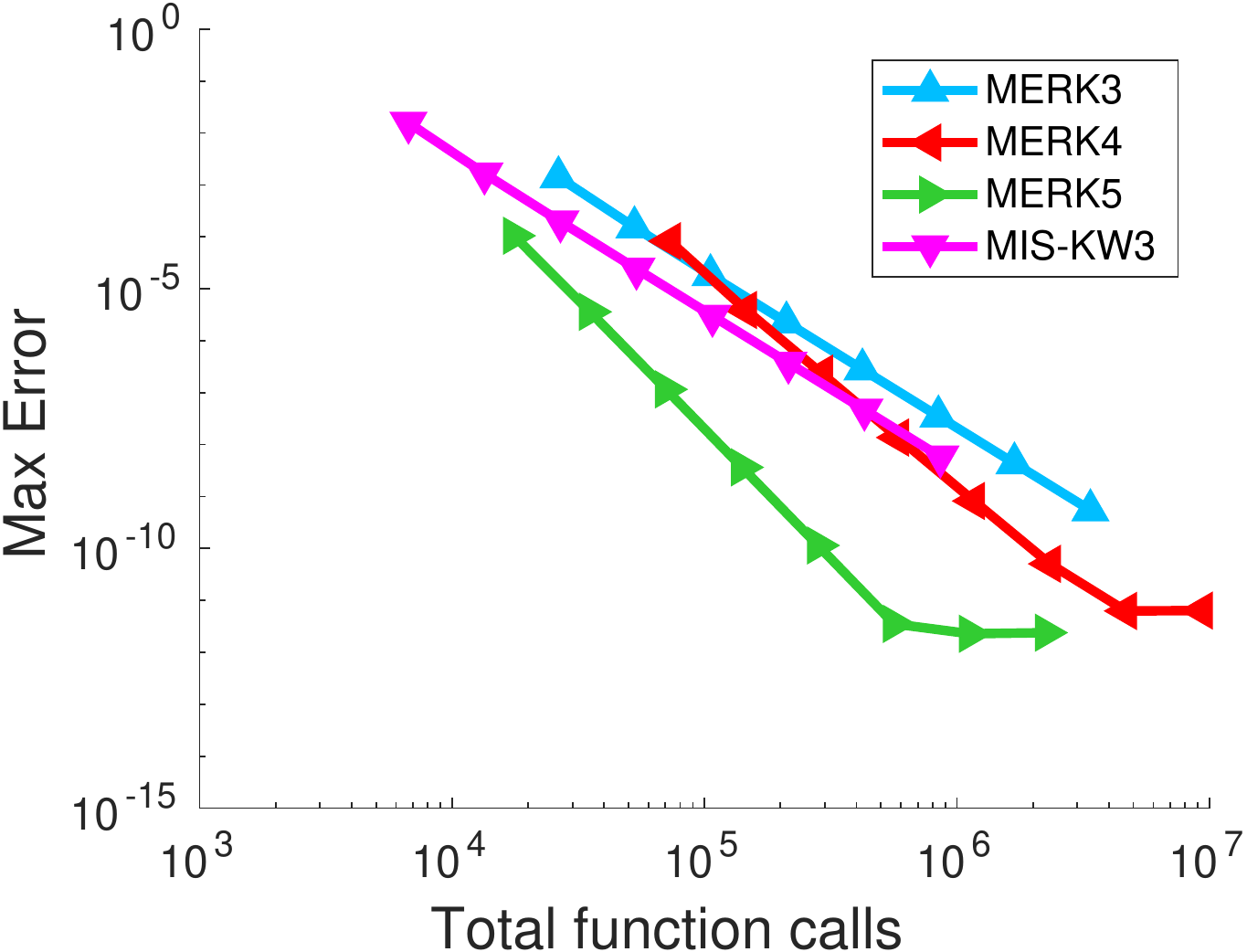}
\caption{}
\label{fig:fastslocoup_tot}
\end{subfigure}
\caption{Bi-directional coupling efficiency. Again, the most efficient
  method depends on how ``cost'' is measured, as well as on the
  desired accuracy, however {\merkfive} demonstrates the best overall
  performance.}
\label{fig:fastslocoup_eff}
\end{figure}

\subsection{Variations in the fast method}
\label{subsec:fast}

We finish by demonstrating the effects of using inner methods with
differing orders of accuracy.  Here, we consider only the
\texttt{MERK} methods, applied to the bi-directional coupling problem
\eqref{eq:bidirectional_pb}.  Here, we vary the order of method
applied for computing both the internal stage solutions \eqref{eq13nc},
$q$, and the step solution \eqref{eq13nd}, $r$.  Recalling the
convergence theory presented in Theorem \ref{theorem2}, a MERK
method of order $p$ should use inner methods of orders $q\ge
p-1$ and $r \ge p$. However, in these tests we apply other variations
on orders to ascertain whether (a) the inner methods could have even lower order
and still obtain overall order $p$, or (b) use of higher-order inner
methods can result in overall convergence higher than $p$.  We present
the best-fit convergence rates for this ensemble of tests in Table
\ref{table:convrates}.

These numerical results show that in fact the inner method order
requirements presented in Theorem \ref{theorem2} are both
necessary and sufficient, i.e., the least-expensive combination for
attaining a MERK method of order $p$ is to compute stage solutions
\eqref{eq13nc} using an inner method of order $p-1$, and the time step
solution \eqref{eq13nd} using an inner method of order $p$.
Furthermore, use of higher-order inner methods with orders $q=r>p$
\emph{does not} result in overall order higher than $p$, due to the
first term $C_1 H^p$ in Theorem \ref{theorem2}, that corresponds to
the coupling between the fast and slow processes.

\begin{table}[h!]
\centering
\begin{tabular}{|c|c|c|c|c|c|c|c|c|}
\hline
\multicolumn{3}{|c|}{{\merkthree}($p=3$)} & \multicolumn{3}{|c|}{{\merkfour}($p=4$)} & \multicolumn{3}{|c|}{{\merkfive}($p=5$)} \\
\hline
$q$ & $r$ & Observed order & $q$ & $r$ & Observed order & $q$ & $r$ & Observed order \\
\hline
2 & 2 & 2.00 & 3 & 3 & 3.01 & 4 & 4 & 4.00 \\
\hline
3 & 2 & 2.00 & 4 & 3 & 3.01 & 5 & 4 & 4.00 \\
\hline
2 & 3 & 3.03 & 3 & 4 & 3.99 & 4 & 5 & 4.97 \\
\hline
3 & 3 & 3.03 & 4 & 4 & 3.99 & 5 & 5 & 4.97 \\
\hline
4 & 4 & 3.03 & 5 & 5 & 3.99 & 6 & 6 & 4.96 \\
\hline
\end{tabular}
\caption{Convergence rate dependence on inner ODE solvers.}
\label{table:convrates}
\end{table}


\section{Conclusion}
\label{sec:conclusion}

We propose a novel class of multirate methods constructed from
explicit exponential Runge--Kutta methods, wherein the action of the
matrix exponential is approximated via solution of ``fast'' initial
value problems for each ExpRK stage.  Algorithmically, these methods
offer a number of desirable properties.  Since these are created
through defining a set of modified IVPs (like (R)MIS and MRI-GARK
methods), MERK implementations have near complete freedom 
in evolving the problem at the fast time scale; however, unlike (R)MIS
and MRI-GARK, MERK methods may utilize inner solvers of reduced
accuracy for the internal stages.  Additionally, since the
MERK structure follows directly from ExpRK methods satisfying Theorem
\ref{theorem1}, derivation of high-order MERK methods, including
versions supporting embeddings for temporal adaptivity, is much
simpler than for alternate multirate frameworks.  As a result, MERK
methods constitute the first multirate algorithms of order five,
without requiring deferred correction or extrapolation techniques.
Furthermore, the proposed approach may be similarly applied to
exponential Rosenbrock methods, allowing for problems where the fast
time scale is nonlinear, although such methods are not considered in
this work.

In addition to proposing the MERK class of multirate methods and
providing rigorous analysis of their convergence, we provide numerical
comparisons of the performance of multiple MERK and MIS methods on a
variety of multirate test problems.  Based on these experiments, we
find that the MERK methods indeed exhibit their theoretical orders of
convergence, including tests that clearly demonstrate our primary
convergence result in Theorem \ref{theorem2}.  Furthermore, the
proposed methods compare favorably against standard MIS multirate
methods, particularly when increased accuracy is desired and for
problems wherein the ``slow'' right-hand side function is
significantly more costly than the ``fast.''

This work may be extended in numerous ways.  As alluded to above,
extensions of these approaches to explicit exponential Rosenbrock
methods are straightforward, and are already under investigation.
Additionally, extensions to higher order will follow from related
developments of higher-order exponential methods.  Finally, we plan to
investigate the use of embeddings at both the fast and slow time
scales to perform temporal adaptivity in both $H$ and $h$ for
efficient, tolerance-based calculations.


\section*{Acknowledgments}
The first author would like to thank Prof.~Hochbruck and Dr.~Demirel for their fruitful discussions during his visit at the Karlsruhe
Institute of Technology (KIT) in 2013 under the support of the DFG Research Training Group
1294 \textquotedblleft Analysis, Simulation and Design of
Nanotechnological Processes\textquotedblright.


\bibliographystyle{siamplain}
\bibliography{references}

\end{document}